\documentclass[a4paper, 10pt]{article}

\usepackage[T1]{fontenc}
\usepackage[latin1]{inputenc}
\usepackage{lmodern}  
\usepackage[english]{babel}
\usepackage{csquotes} 

\usepackage[scale=0.73, a4paper]{geometry}
\usepackage{textcomp} 
\usepackage{graphicx} 
\usepackage{amsmath} 
\usepackage{amsthm} 
\usepackage{amssymb} 
\usepackage{amsfonts}
\usepackage{mathtools}
\usepackage{thmtools} 
\usepackage{mathrsfs} 
\usepackage{mathdots} 
\usepackage{wasysym} 
\usepackage{extpfeil} 
\usepackage{verbatim} 
\usepackage{longtable} 
\usepackage{bigstrut} 
\usepackage[table, dvipsnames]{xcolor} 
\usepackage{floatpag} 
\usepackage{enumitem} 
\usepackage{xparse} 
\usepackage{mathscinet} 
\usepackage{adjustbox} 
\usepackage[obeyDraft]{todonotes} 

\usepackage[final,                  
            colorlinks,             
            plainpages = false,     
            linktoc=all,            
            linkcolor=black,        
            citecolor=black,         
            urlcolor=black,          
            filecolor=black         
            ]{hyperref}

\usepackage[citestyle=alphabetic,   
            bibstyle=alphabetic,    
            maxbibnames=50,         
            maxcitenames=3,         
            autocite=inline,        
            block=space,            
            hyperref=true,          
            backref=false,           
            backrefstyle=three,     
            date=short,             
            arxiv=pdf,              
            isbn=false,             
            url=false,              
            doi=false,               
            eprint=true,            
            giveninits=true,        
            sorting=nyt,            
            backend=bibtex8,
            ]{biblatex}

\bibliography{abc}

\hypersetup{pdftitle = {ABC of p-Cells},
            pdfauthor = {Lars Thorge Jensen},
            pdfkeywords = {p-canonical basis} {Hecke algebra},
            pdfdisplaydoctitle      
            }
            
\usepackage[arrow, matrix, curve]{xy} 
\usepackage{tikz} 
\usetikzlibrary{arrows, arrows.meta, bending, fit, intersections, calc, external, shapes.misc, shapes.geometric, decorations.markings, decorations.pathreplacing}
\usepackage{pgfplots}
\pgfplotsset{compat=1.14}

\allowdisplaybreaks 

\setkeys{Gin}{draft=false}
            
\declaretheoremstyle[
spaceabove=\topsep, spacebelow=\topsep,
headfont=\normalfont\bfseries,
notefont=\mdseries, notebraces={(}{)},
bodyfont=\itshape,
postheadspace=\newline
]{break}

\declaretheoremstyle[
spaceabove=\topsep, spacebelow=\topsep,
headfont=\normalfont\bfseries,
notefont=\mdseries, notebraces={}{},
bodyfont=\itshape,
postheadspace=\newline
]{refbreak}

\declaretheorem[title=Theorem, style=plain, numberwithin=section]{thm}
\declaretheorem[title=Proposition, style=plain, numberlike=thm]{prop}
\declaretheorem[title=Lemma, style=plain, numberlike=thm]{lem}
\declaretheorem[title=Corollary, style=plain, numberlike=thm]{cor}

\declaretheorem[title=Assumption, style=plain, numberlike=thm]{ass}

\declaretheorem[title=Theorem, style=break, numberlike=thm]{thmlab}

\declaretheorem[title=Conjecture, style=break, numberlike=thm]{conj}

\declaretheorem[title=Definition, style=definition, numberlike=thm]{defn}

\declaretheorem[title=Remark, style=remark, numberlike=thm]{remark}

\declaretheorem[title=Example, style=remark, numberlike=thm]{ex}

\usepackage[capitalize]{cleveref} 
                          
\crefname{thm}{Theorem}{Theorems}
\crefname{prop}{Proposition}{Propositions}
\crefname{lem}{Lemma}{Lemmata}
\crefname{cor}{Corollary}{Corollaries}
\crefname{rem}{Reminder}{Reminders}
\crefname{defn}{Definition}{Definitions}

\crefname{thmlab}{Theorem}{Theorems}
\crefname{proplab}{Proposition}{Propositions}
\crefname{lemlab}{Lemma}{Lemmata}
\crefname{corlab}{Corollary}{Corollaries}
\crefname{remlab}{Reminder}{Reminders}
\crefname{conj}{Conjecture}{Conjectures}

\crefname{thmreflab}{Theorem}{Theorems}
\crefname{propreflab}{Proposition}{Propositions}
\crefname{lemreflab}{Lemma}{Lemmata}
\crefname{correflab}{Corollary}{Corollaries}
\crefname{remreflab}{Reminder}{Reminders}
\crefname{conjref}{Conjecture}{Conjectures}

\crefname{remark}{Remark}{Remarks}
\crefname{claim}{Claim}{Claims}
\crefname{ex}{Example}{Examples}

\crefname{section}{Section}{Sections}
\crefname{figure}{Figure}{Figures}
\crefname{equation}{}{}
\crefname{ass}{Assumption}{Assumptions}

\CompileMatrices

\entrymodifiers={+!!<0pt,\fontdimen22\textfont2>} 

\def\clap#1{\hbox to 0pt{\hss#1\hss}}

\makeatletter
\def\underbracket{%
    \@ifnextchar[{\@underbracket}{\@underbracket [\@bracketheight]}%
}
\def\@underbracket[#1]{%
    \@ifnextchar[{\@under@bracket[#1]}{\@under@bracket[#1][0.4em]}%
}
\def\@under@bracket[#1][#2]#3{
    \mathop{\vtop{\m@th \ialign {##\crcr $\hfil \displaystyle {#3}\hfil $%
    \crcr \noalign {\kern 3\p@ \nointerlineskip }\upbracketfill {#1}{#2}
    \crcr \noalign {\kern 3\p@ }}}}\limits}

\def\upbracketfill#1#2{$\m@th \setbox \z@ \hbox {$\braceld$}
    \edef\@bracketheight{\the\ht\z@}\bracketend{#1}{#2}
    \leaders \vrule \@height #1 \@depth \z@ \hfill
    \leaders \vrule \@height #1 \@depth \z@ \hfill \bracketend{#1}{#2}$}

\def\bracketend#1#2{\vrule height #2 width #1\relax}
\makeatother

\makeatletter
\def\thmt@refnamewithcomma #1#2#3,#4,#5\@nil{%
  \@xa\def\csname\thmt@envname #1utorefname\endcsname{#3}%
  \ifcsname #2refname\endcsname
    \csname #2refname\expandafter\endcsname\expandafter{\thmt@envname}{#3}{#4}%
  \fi
}
\makeatother

\makeatletter
\newcommand*\rel@kern[1]{\kern#1\dimexpr\macc@kerna}
\newcommand*\widebar[1]{%
  \begingroup
  \def\mathaccent##1##2{%
    \rel@kern{0.8}%
    \overline{\rel@kern{-0.8}\macc@nucleus\rel@kern{0.2}}%
    \rel@kern{-0.2}%
  }%
  \macc@depth\@ne
  \let\math@bgroup\@empty \let\math@egroup\macc@set@skewchar
  \mathsurround\z@ \frozen@everymath{\mathgroup\macc@group\relax}%
  \macc@set@skewchar\relax
  \let\mathaccentV\macc@nested@a
  \macc@nested@a\relax111{#1}%
  \endgroup
}
\makeatother

\makeatletter
\newcommand{\subjclass}[2][1991]{%
  \let\@oldtitle\@title%
  \gdef\@title{\@oldtitle\footnotetext{#1 \emph{Mathematics subject classification.} #2}}%
}
\newcommand{\keywords}[1]{%
  \let\@@oldtitle\@title%
  \gdef\@title{\@@oldtitle\footnotetext{\emph{Key words and phrases.} #1.}}%
}
\makeatother

\makeatletter
\newcommand{\extp}{\@ifnextchar^\@extp{\@extp^{\,}}}
   \def\@extp^#1{\mathop{\bigwedge\nolimits^{\!#1}}}
\makeatother

\DeclareMathOperator{\Hom}{Hom}
\DeclareMathOperator{\End}{End}

\DeclareMathOperator{\id}{Id}

\DeclareMathOperator{\ch}{ch}
\DeclareMathOperator{\Ex}{Ex}
\DeclareMathOperator{\df}{df}

\DeclareMathOperator{\grk}{grk}

\DeclareMathOperator{\f}{f}

\DeclareMathOperator{\BS}{BS}

\DeclarePairedDelimiter\floor{\lfloor}{\rfloor}

\newcommand{\Top}{\rule{0pt}{2.6ex}}       

\newcommand{\defeq}{\ensuremath{\coloneqq}}

\newcommand{\cat}[1]{\ensuremath{\mathcal{#1}}}

\newcommand{\karalg}[1]{\ensuremath{\cat{K}ar(#1)}}

\newcommand{\G}[1]{\ensuremath{[\cat{#1}]}}
\newcommand{\Galg}[1]{\ensuremath{[#1]}}

\newcommand{\Z}{\ensuremath{\mathbb{Z}}}
\newcommand{\N}{\ensuremath{\mathbb{N}}}
\newcommand{\Q}{\ensuremath{\mathbb{Q}}}
\newcommand{\R}{\ensuremath{\mathbb{R}}}
\newcommand{\C}{\ensuremath{\mathbb{C}}}
\newcommand{\F}[1]{\ensuremath{\mathbb{F}_{#1}}}
\newcommand{\K}{\ensuremath{\mathbb{K}}}

\newcommand{\p}[1]{\ensuremath{\prescript{p}{}{#1}}}
\newcommand{\pre}[2]{\ensuremath{{}^{#1} #2}}

\newcommand{\heck}[1][]{\ensuremath{\mathcal{H}_{#1}}}

\newcommand{\std}[2][H]{\ensuremath{#1_{#2}}}
\newcommand{\kl}[2][H]{\ensuremath{\underline{#1}_{#2}}}
\DeclareDocumentCommand{\pcan}{O{p} O{H} m}{\ensuremath{\prescript{#1}{}{\underline{#2}}_{#3}}}
\newcommand{\desc}[1]{\ensuremath{\mathcal{#1}}}

\newcommand{\expr}[1]{\ensuremath{\underline{#1}}}
\newcommand{\mult}[1]{\ensuremath{\expr{#1}_{\bullet}}}
\newcommand{\cle}[2][p]{\ensuremath{\, \overset{#1}{\underset{#2}{\leqslant}} \,}}
\newcommand{\cge}[2][p]{\ensuremath{\, \overset{#1}{\underset{#2}{\geqslant}} \,}}
\newcommand{\multLe}[2][p]{\ensuremath{\, \overset{#1}{\underset{#2}{\leftarrow}} \,}}
\newcommand{\multGe}[2][p]{\ensuremath{\, \overset{#1}{\underset{#2}{\rightarrow}} \,}}
\newcommand{\clt}[2][p]{\ensuremath{\, \overset{#1}{\underset{#2}{<}} \,}}

\newcommand{\ceq}[2][p]{\ensuremath{\, \overset{#1}{\underset{#2}{\sim}} \,}}

\newcommand{\T}{\ensuremath{\cat{T}}}

\newcommand{\rt}[1][]{\ensuremath{\alpha_{#1}}}
\newcommand{\cort}[1][]{\ensuremath{\alpha_{#1}^{\vee}}}

\newcommand{\charlat}{\ensuremath{X}}
\newcommand{\cocharlat}{\ensuremath{X^{\vee}}}

\newcommand{\Wid}{\ensuremath{e}}

\newcommand{\HC}[1][]{\ensuremath{\mathbf{H}^{#1}}}

\newcommand{\HCnl}[1]{\ensuremath{\HC[\nless #1]}}

\newcommand{\Homnl}[2][]{
   \ifthenelse{ \equal{#1}{} } { \ensuremath{\Hom_{\nless #2}} } 
   { \ensuremath{\Hom_{\nless #2, #1}} } }
\newcommand{\Endnl}[2][]{
   \ifthenelse{ \equal{#1}{} } { \ensuremath{\End_{\nless #2}} } 
   { \ensuremath{\End_{\nless #2, #1}} } }
\newcommand{\D}{\ensuremath{\mathbb{D}}}

\newcommand{\Address}{
  \bigskip{\footnotesize
  \textsc{Universit\'{e} Clermont Auvergne, CNRS, LMBP, F-63000 Clermont-Ferrand, France}\par\nopagebreak
  \textit{E-mail address}, Lars~Thorge~Jensen: \texttt{lars\_thorge.jensen@uca.fr}
}}

\tikzset{%
  DynNode/.style={circle, inner sep=2pt, draw=black, fill=white},
  Greater/.style={pos=0.65, inner sep=0mm, outer sep=0mm},
  highlight/.style={rectangle,rounded corners,fill=red!15,draw=red,
    fill opacity=0.5,thick},
  root/.style={draw, color=black, thick, ->},
  plane/.style={draw, color=black, very thin},
  origin/.style={fill, color=black},
  sline/.style={color=Red, thick},
  tline/.style={color=NavyBlue, thick},
  uline/.style={color=Goldenrod, thick},
  bendBelow/.style={bend left=70, looseness=2},
  bendAbove/.style={bend right=70, looseness=2},
  object/.style={circle, fill, inner sep=1.5pt, outer sep=0mm},
  labelling/.style={outer sep=0mm, inner sep=0mm},
  1morph/.style={->, shorten >= 0.5pt, >=stealth'},
  2morph/.style={-implies,double,double equal sign distance,
                 shorten >=2pt, shorten <=3pt},
  spot/.style={color=black, thin, dashed},
  s/.style={color=Red},
  t/.style={color=NavyBlue},
  u/.style={color=Goldenrod},
  line/.style={draw, line width=2pt},
  dot/.style={fill, thin},
  sph/.style={fill, color=black!20, opacity=0.5},
  squig/.style={decoration=snake, decorate, ->},
  root/.style={draw, color=black, line width=2pt, ->},
  myptr/.style={decoration={markings,mark=at position 1 with %
    {\arrow[scale=3,>=stealth]{>}}},postaction={decorate}},
  on each segment/.style={
    decorate,
    decoration={
      show path construction,
      moveto code={},
      lineto code={
        \path [#1]
        (\tikzinputsegmentfirst) -- (\tikzinputsegmentlast);
      },
      curveto code={
        \path [#1] (\tikzinputsegmentfirst)
        .. controls
        (\tikzinputsegmentsupporta) and (\tikzinputsegmentsupportb)
        ..
        (\tikzinputsegmentlast);
      },
      closepath code={
        \path [#1]
        (\tikzinputsegmentfirst) -- (\tikzinputsegmentlast);
      },
    },
  },
  mid arrow/.style={postaction={decorate,decoration={
        markings,
        mark=at position .5 with {\arrow[#1]{stealth}}
      }}},
  sline/.style={draw, line width=1pt, postaction={on each segment={mid arrow=black}}},
  sregion/.style={fill, opacity=0.2},
  st/.style={fill=Fuchsia},
  su/.style={fill=YellowOrange},
  tu/.style={fill=ForestGreen},
  clabel/.style={fill=none, red}, 
  str/.style={<->}
}

\newcommand{\BigFig}[1]{\parbox{12pt}{\Huge #1}}
\newcommand{\BigZero}{\BigFig{0}}
\newcommand{\BigAst}{\BigFig{$\ast$}}

\begin{document}

\def \dist {0.3cm}
\def \xdist {0.4cm}
\def \ydist {0.5cm}
\def \circSize {2.5pt} 
\def \armLen {0.4cm} 
\def \edgeShift {0.5mm} 
\def \wingLen {0.3cm} 
\def \sxdist {0.4cm}

\title{The ABC of p-Cells}
\author{Lars Thorge Jensen}
\date{} 

\maketitle

\begin{abstract}  
  Parallel to the very rich theory of Kazhdan-Lusztig cells in characteristic $0$, 
  we try to build a similar theory in positive characteristic. We study cells
  with respect to the $p$-canonical basis of the Hecke algebra of a crystallographic
  Coxeter system (see \cite{JW}). Our main technical tool are the  star-operations 
  introduced by Kazhdan-Lusztig in \cite{KL} which have interesting numerical
  consequences for the $p$-canonical basis. As an application, we explicitely
  describe $p$-cells in finite type $A$ (i.e. for symmetric groups) using 
  the Robinson-Schensted correspondence. Moreover, we show that Kazhdan-Lusztig 
  cells in finite types $B$ and $C$ decompose into $p$-cells for $p > 2$.
\end{abstract}

\tableofcontents

\vspace{1cm}

\section{Introduction}

The Hecke algebra of a crystallographic Coxeter system admits several 
geometric or algebraic categorifications (see \cite{KL, EW2}). 
In these cases the resulting canonical basis gives rise to the 
famous Kazhdan-Lusztig basis (see \cite{KL2, EW1}) in the characteristic
$0$ setting. The original motivation for the Kazhdan-Lusztig basis 
was to explicitly construct representations of the Hecke algebra 
(see \cite[comments on {[37]}]{LuPapers}). This naturally led 
Kazhdan and Lusztig to study cells in the Hecke algebra with respect 
to the Kazhdan-Lusztig basis and the corresponding cell modules.

In \cite{JW} the $p$-canonical basis for the Hecke algebra of a crystallographic 
Coxeter system was introduced. It should be thought of as a positive characteristic 
analogue of the Kazhdan-Lusztig basis. The $p$-canonical basis shares strong 
positivity properties with the Kazhdan-Lusztig basis (similar to the ones 
described by the Kazhdan-Lusztig positivity conjectures), but it loses many 
of its combinatorial properties. For this reason, it is much harder to compute 
the $p$-canonical basis which is only known in small examples.

Even without explicit knowledge of the $p$-canonical basis, one may obtain a first 
approximate description of the multiplicative structure by studying the 
left, right or two-sided cell preorder with respect to the $p$-canonical basis.
Replacing the Kazhdan-Lusztig basis by the $p$-canonical basis in the definition
of the left (resp. right or two-sided) cells leads to the notion of left (resp. 
right or two-sided) $p$-cells. In this paper, we introduce $p$-cell theory which
we hope will eventually turn out to be as influential as Kazhdan-Lusztig cell
theory. It is current work of progress to describe $p$-cells in affine Weyl
groups of small rank. For affine Weyl groups $p$-cells are closely related to
tensor ideals of tilting modules (see for example \cite{AnCells, AHRHumphreysConj,
AHRCellConjs}).

The first properties of $p$-cells that we prove in \cref{secpCellTheory} are 
the following: Left and right $p$-cells are related by taking inverses 
(see \cref{corLeftRightEq}), just like for Kazhdan-Lusztig cells. The set 
of elements with a fixed left descent set decomposes into right $p$-cells 
(see \cref{lemCleDesc}). We also study which automorphisms of the Hecke 
algebra are well-behaved with respect to the $p$-canonical basis (see 
\cref{propCoxAutom}). The most important result of this section is a 
certain compatibility of $p$-cells with parabolic subgroups:
We show that any right $p$-cell preorder relation in a finite, standard 
parabolic subgroup $W_I$ induces right $p$-cell preorder relations in 
each right $W_I$-coset (see \cref{thmParaCombPCells}).

In \cite[Theorem 1.4]{KL}, Kazhdan and Lusztig show that in type $A$ the cell modules give 
the irreducible modules of the Hecke algebra $\heck$ for generic parameter $v$.
In their proof, they introduce the Kazhdan-Lusztig star-operations, generalizing
the (dual) Knuth operations for symmetric groups to pairs of simple reflections
$s$ and $t$ in a Coxeter group with $st$ of order $3$. The study of the consequences
of the star-operations for the structure coefficients leads to an explicit
description of the Kazhdan-Lusztig cells in symmetric groups (see \cite[\S5]{KL}).
The Kazhdan-Lusztig cells in type $A$ can be characterized via the Robinson-Schensted
correspondence (see \cite{ArRSLCells}), which gives a bijection $w \mapsto (P(w), Q(w))$
between the symmetric group $S_n$ and pairs of standard tableaux of the same shape
with $n$ boxes: The Kazhdan-Lusztig right cell of a given element $w \in S_n$
is given by the set of elements in $S_n$ that have the same $P$-symbol as $w$
under the Robinson-Schensted correspondence.

In an attempt to explicitly describe $p$-cells for symmetric groups, we
studied the consequences of the Kazhdan-Lusztig star-operations for the 
$p$-canonical basis. We deduce many interesting relations for the structure 
coefficients of the $p$-canonical basis and for the base change coefficients 
between the $p$-canonical and the Kazhdan-Lusztig basis. The reader is
invited to compare these with the results in \cite[\S4]{KL} 
and \cite[\S10.4]{LuCellsInAffWeylGrpsI}. If a $p$-canonical basis element 
differs from the corresponding Kazhdan-Lusztig basis element, then these 
identities often allow to deduce the non-triviality of other $p$-canonical 
basis elements (see \cite[Remark 5.2. (9)]{LuWTiltCharsSL3} for an example). 
Another consequence is that $p$-cells coincide with 
Kazhdan-Lusztig cells in symmetric groups for all primes $p$ (see 
\cref{thmPCellsTypeA}) and in particular that they are independent of $p$. This is 
particularly interesting since the $p$-canonical basis does differ from the 
Kazhdan-Lusztig basis for many primes $p$ (see \cite{WTorsion}).

Thus, in type $A$ the $p$-canonical basis of each cell module gives after 
extending scalars to $\C$ and specializing $v$ to $1$ a basis of the 
corresponding complex irreducible representation of the symmetric group. 
Letting $p$ vary, we obtain a very interesting family of bases that merits further study.
The relation between Specht modules, the Kazhdan-Lusztig cell representations, 
and their corresponding natural bases was further studied for the Hecke algebra as well as 
for the group algebra of symmetric groups in \cite{MaWgraphs, GMCellsTypeA, MPSpechtModsTypeA}.

Supported by the results in finite type $A$, one may hope that Kazhdan-Lusztig
cells always decompose into $p$-cells. Unfortunately, this is not the case
as we show in \cref{secDecompCounterex}. However, we believe that 
the corresponding statement may still be true for $p$ good for the
corresponding algebraic group. In \cref{secDecompCriterion}, we develop
a simple criterion when Kazhdan-Lusztig right cells decompose into right $p$-cells,
which reduces the question to minimal elements with respect to the weak right
Bruhat order in each cell.

In a series of papers \cite{GaPrimIdealsI, GaPrimIdealsII, GaPrimIdealsIII, GaPrimIdealsIV}, 
Garfinkle generalizes the Robinson-Schensted correspondence to types $B$, $C$ and $D$. 
She develops combinatorial algorithms to associate to a Weyl group element $w$ 
a pair $(T_L(w), T_R(w))$ of standard domino tableaux of the same shape from which 
$w$ can be reconstructed. (Note that the definition of a standard domino tableau 
depends on the type.) The main difference to the situation in type $A$ is that 
the partition of the Weyl group into sets with the same left domino tableau 
is finer than the partition into Kazhdan-Lusztig left cells (see 
\cite[\S3]{McGCellsAndDominoTableaux}). For this reason, Garfinkle groups 
the set of dominoes in a tableau into \emph{cycles} and classifies them as 
``open'' or ``closed''. For each open cycle, she defines an involutive algorithm
called ``moving a tableau through an open cycle'' that changes only the positions
of the dominoes in the open cycle. Based on this, she defines an equivalence relation
on the set of standard domino tableaux by declaring two to be equivalent if one
can be obtained from the other by moving through a particular subset of 
open cycles. (In type $B/C$ only the open cycle containing the domino with label
$1$ is excluded.) One of the main results is that two elements in a Weyl 
group of type $B/C$ lie in the same Kazhdan-Lusztig left cell if and 
only if their corresponding left standard domino tableaux are equivalent 
(see \cite[Corollary 3.5.6. and Theorem 3.5.9.]{GaPrimIdealsIII}).
Unpublished work of Garfinkle (see \cite{GaPrimIdealsIV}), Pietraho and McGovern
extends this result to type $D$.
Moreover, Garfinkle
shows that a generalization of Vogan's $\tau$-invariant gives a complete
invariant for Kazhdan-Lusztig left cells. Based on our results on the Kazhdan-Lusztig
star operations, we can show that the equivalence classes with respect to this
generalized $\tau$-invariant give a refinement of the left $p$-cells under small
assumptions on $p$. This implies, that Kazhdan-Lusztig left cells decompose
into left $p$-cells in finite types $B$ and $C$ for $p > 2$.

\subsection{Structure of the Paper}

\begin{description}
   \item[\Cref{secBack}] We introduce notation and recall important results about
         the Hecke algebra and the $p$-canonical basis.
   \item[\Cref{secpCellTheory}] We define $p$-cells and prove some of their elementary
         properties. The most important results are the compatibility of $p$-cells with
         parabolic subgroups and a criterion for Kazhdan-Lusztig cells to decompose into $p$-cells.
         We also give interesting examples of $p$-cells and state a conjecture
         resulting from extensive computer calculations.
   \item[\Cref{secStarOps}] We introduce the Kazhdan-Lusztig star operations.
         Then we study in detail consequences for base change and structure coefficients
         of the $p$-canonical basis and for $p$-cells. After introducing Vogan's 
         generalized $\tau$-invariant, we show that left $p$-cells give a refinement
         of the $\tau$-equivalence classes under small assumptions on $p$. In the end,
         we show that $p$-cells in finite type $A$ are given by the Robinson-Schensted
         correspondence.
\end{description}

\subsection{Acknowledgements}
Since this paper is part of the author's PhD-thesis, I would like to thank 
my supervisor Geordie Williamson - his constant support, his inspiring 
mathematical vision and his enthusiasm were crucial for the success of my 
PhD project. 
Moreover, I am grateful to the Max-Planck Institute for Mathematics for the perfect working conditions and the financial support for my research visit in Sydney.
The project was started at the School of Mathematics and Statistics at the University
of Sydney and finished at the Max-Planck Institute for Mathematics in Bonn.
I would also like to thank Simon Riche for detailed comments and Monty McGovern for
providing a preliminary version of \cite{GaPrimIdealsIV} and of his joint work 
in progress with Thomas Pietraho.

\section{Background}
\label{secBack}

\subsection{Coxeter Systems and Generalized Cartan Matrices}
\label{secCox}
Let $S$ be a finite set and $(m_{s, t})_{s, t \in S}$ be a matrix with entries in
$\N \cup \{\infty\}$ such that $m_{s,s} = 1$ and $ m_{s, t} = m_{t, s} \geqslant 2$ 
for all $s \neq t \in S$. Denote by $W$ the group generated by $S$ subject to the relations 
$(st)^{m_{s, t}} = 1$ for $s, t \in S$ with $m_{s, t} < \infty$. We say that $(W, S)$
is a \emph{Coxeter system} and $W$ is a \emph{Coxeter group}. The Coxeter group $W$
comes equipped with a length function $l: W \rightarrow \N$ and the Bruhat order
$\leqslant$ (see \cite{Hum} for more details). A Coxeter system $(W, S)$ is 
called \emph{crystallographic} if $m_{s, t} \in \{2, 3, 4, 6, \infty\}$ for 
all $s \neq t \in S$. We denote the identity of $W$ by $\Wid$. For $w \in W$ we
define its \emph{left descent set} via
\[ \desc{L}(w) \defeq \{ s \in S \; \vert \; l(sw) < l(w) \} \text{.} \]
The \emph{right descent set} of $w$ is given by $\desc{R}(w) \defeq \desc{L}(w^{-1})$.

Define an \emph{expression} to be a finite sequence of elements in $S$. We denote by
\[ \Ex(S) \defeq \{\varnothing\} \cup \bigcup_{i \in \N \setminus \{0\}} \underbrace{S \times 
   \dots \times S}_{i\text{-times}} \]
the set of all expressions in $S$. For an expression $\expr{w} = (s_1, s_2, \dots, s_n)$
denote its \emph{length} by $l(\expr{w}) = n$. For two expressions $\expr{x}$ and 
$\expr{y}$ in $S$, we write $\expr{x}^{\frown}\expr{y}$ for their concatenation.
The multiplication gives a canonical map $\Ex(S) \rightarrow W$, 
$\expr{w} \mapsto \mult{w}$. An expression $\expr{w}$ in $S$
is called \emph{reduced} if $l(\expr{w}) = l(\mult{w})$. For an expression 
$\expr{w} = (s_1, s_2, \dots, s_n)$ in $S$ a \emph{subexpression} of $\expr{w}$ is
a sequence $\expr{w}^{\expr{e}} = (s_1^{e_1}, s_2^{e_2}, \dots, s_n^{e_n})$ 
where $e_i \in \{0, 1\}$ for all $i$. The sequence $\expr{e} = (e_1, e_2, \dots, e_k)$ 
is called the \emph{associated $01$-sequence}. We usually decorate $\expr{e}$ as follows: 
For $1 \leqslant k \leqslant n$ denote by $\expr{w}_{\leqslant k} \defeq (s_1, s_2, \dots, s_k)$ 
the first $k$ terms and set $w_k \defeq 
(\expr{w}_{\leqslant k}^{\expr{e}_{\leqslant k}})_{\bullet}$. Assign to $e_i$ a 
decoration $d_i \in \{U, D\}$ where $U$ stands for \emph{Up} and $D$ for 
\emph{Down} as follows:
\[ d_1 \defeq U \text{ and } d_i \defeq 
   \begin{cases}
      U & \text{if } w_{i-1} s_i > w_{i-1} \text{,} \\
      D & \text{if } w_{i-1} s_i < w_{i-1} 
   \end{cases}
   \text{ for } 2 \leqslant i \leqslant n\text{.}
\]
We often write the decorated sequence as $(d_1 e_1, d_2 e_2, \dots, d_n e_n)$.
The sequence of elements $e$, $w_1$, $w_2$, $\dots$, $w_n$ is called the \emph{Bruhat stroll}
associated to $\expr{w}^{\expr{e}}$. The \emph{defect} of $\expr{e}$ is defined to be
\[ \df(e) \defeq \vert \{ i \; \vert \; d_i e_i = U0 \}\vert - 
            \vert \{ i \; \vert \; d_i e_i = D0 \} \vert \text{.} \]
\begin{ex}
   To illustrate the definitions, consider for example the case $S = \{s, t\}$ 
   with $m_{s,t} = m_{t, s} = 3$ (i.e. type $A_2$). The reduced expression $(s, t, s)$
   admits two decorated $01$-sequences expressing $s$:
   \begin{align*}
      &(U1, U0, D0) \quad \text{of defect } 0\\
      &(U0, U0, U1) \quad \text{of defect } 2\\
   \end{align*}
\end{ex}

From now on, fix a generalized Cartan matrix $A = (a_{i,j})_{i, j \in J}$ (see 
\cite[\S1.1]{TiGrpKM}). Let $(J, \charlat, \{ \rt[i] : i \in J\}, 
\{ \cort[i] : i \in J \})$ be an associated Kac-Moody root datum (see 
\cite[\S1.2]{TiGrpKM} for the definition). Then $X$ is a finitely generated 
free abelian group, and for $i \in J$ we have elements $\rt[i]$ and $\cort[i]$ 
of $X$ and $\cocharlat = \Hom_{\Z}(\charlat, \Z)$ respectively that satisfy
$a_{i,j} = \cort[i](\rt[j])$ for all $i, j \in J$.

To $A$ we associate a crystallographic Coxeter system $(W, S)$ as follows:
Choose a set of simple reflections $S$ of cardinality $\lvert J \rvert$ and
fix a bijection $S \overset{\sim}{\rightarrow} J$, $s \mapsto i_s$. For $s \ne t \in S$
we define $m_{s, t}$ to be $2$, $3$, $4$, $6$, or $\infty$ if $a_{i_s, j_s} a_{j_s, i_s}$
is $0$, $1$, $2$, $3$, or $\geqslant 4$ respectively.

Fix a commutative ring $k$. In both cases, ${}^k V \defeq \cocharlat \otimes_{\Z} k$ 
yields a balanced, potentially non-faithful realization of the Coxeter system over $k$. 
Set ${}^k V^{\ast} \defeq \Hom_{k}({}^k V, k)$ and note that ${}^k V^{\ast}$ is 
isomorphic to $\charlat \otimes_{\Z} k$. A realization obtained in this way is called
a \emph{Cartan realization} (see \cite[\S10.1]{AMRWFreeMonodromicMixedTiltSheaves}).
Throughout, we will assume our realization to satisfy:
%
%
%
\begin{ass}[Demazure Surjectivity]
   \label{assDemazureSurj}
   The maps $\rt[s]: {}^k V \rightarrow k$ and $\cort[s]: {}^k V^{\ast} 
   \rightarrow k$ are surjective for all $s \in S$.
\end{ass}
\noindent This is automatically satisfied if $2$ is invertible in $k$ or if the
Coxeter system $(W, S)$ is of simply-laced type and of rank $\lvert S \rvert \geqslant 2$.

We denote by $R = S({}^k V^{\ast})$ the symmetric algebra of ${}^k V^{\ast}$
over $k$ and view it as a graded ring with ${}^k V^{\ast}$ in degree $2$. 
Given a graded $R$-bimodule $B = \bigoplus_{i\in \Z} B^i$, we denote by 
$B(1)$ the shifted bimodule with $B(1)^i = B^{i+1}$.


\subsection{The Hecke Algebra}
\label{secHeck}
The Hecke algebra is the free $\Z[v, v^{-1}]$-algebra with $\{ \std{w} \; \vert \; 
w \in W \}$ as basis, called the \emph{standard basis}, and multiplication determined by:
\begin{alignat*}{2}
   \std{s}^2 &= (v^{-1} - v) \std{s} + 1  \qquad && \text{for all } s \in S \text{,} \\
   \std{x} \std{y} &= \std{xy} && \text{if } l(x) + l(y) = l(xy) \text{.}
\end{alignat*}

There is a unique $\Z$-linear involution $\widebar{(-)}$ on $\heck$ satisfying
$\widebar{v} = v^{-1}$ and $\widebar{\std{x}} = \std{x^{-1}}^{-1}$. The Kazhdan-Lusztig 
basis element $\kl{x}$ is the unique element in $\std{x} + \sum_{y < x} v\Z[v] H_y$ 
that is invariant under $\widebar{(-)}$. This is Soergel's normalization from 
\cite{S2} of a basis introduced originally in \cite{KL}. For a sequence 
$\expr{w} = (s_1, s_2, \dots, s_n)$ in $S$, we write $\kl{\expr{w}}$ for the element
$\kl{s_1}\kl{s_2} \dots \kl{s_n}$.  After fixing a reduced expression $\expr{w}$ of every
element $w \in W$, the set $ \{ \kl{\expr{w}} \; \vert \; w \in W\}$
gives a basis of $\heck$, called the \emph{Bott-Samelson basis}.

Let $\iota$ be the $\Z[v, v^{-1}]$-linear anti-involution on $\heck$ satisfying 
$\iota(\std{s}) = \std{s}$ for $s \in S$ and thus $\iota(\std{x}) = \std{x^{-1}}$.

\subsection{The Diagrammatic Category of Soergel Bimodules}

In this section, we introduce the diagrammatic category of Soergel bimodules. 
The main reference for this is \cite{EW2} (see also \cite{E3} in the dihedral 
case and \cite{EKh} in type $A$).

Let $\BS$ be the diagrammatic category of Bott-Samelson bimodules
as introduced in \cite[\S2.3]{JW}. It is a diagrammatic, strict monoidal 
category enriched over $\Z$-graded left $R$-modules.

Let $\HC$ be the Karoubian envelope of the graded version of the additive closure
of $\BS$, in symbols $\HC = \karalg{\BS}$. We call $\HC$ the 
\emph{diagrammatic category of Soergel bimodules}. In other words, in the passage 
from $\BS$ to $\HC$ we first allow direct sums and grading shifts (restricting 
to degree preserving homomorphisms) and then the taking of direct summands. The 
following properties can be found in \cite[Lemma 6.24, Theorem 6.25 and 
Corollary 6.26]{EW2}:

\begin{thmlab}[Properties of $\HC$]
   \label{thmDiagProps}
   Let $k$ be a complete, local, integral domain (e.g. a field or the $p$-adic 
   integers $\Z_p$).
   \begin{enumerate}
      \item $\HC$ is a Krull-Schmidt category.
      \item For all $w \in W$ there exists a unique indecomposable object $\pre{k}{B}_w 
            \in \HC$ which is a direct summand of $\expr{w}$ for any reduced 
            expression $\expr{w}$ of $w$ and which is not isomorphic to a grading 
            shift of any direct summand of any expression $\expr{v}$ for $v < w$. 
            In particular, the object $\pre{k}{B}_w$ does not depend up to isomorphism 
            on the reduced expression $\expr{w}$ of $w$.
      \item The set $\{ \pre{k}{B}_w \; \vert \; w \in W\}$ gives a complete set of 
            representatives of the isomorphism classes of indecomposable objects 
            in $\HC$ up to grading shift.
      \item There exists a unique isomorphism of $\Z[v, v^{-1}]$-algebras
            \[ \ch: \Galg{\HC} \longrightarrow \heck \]
            sending $[\pre{k}{B}_s]$ to $\kl{s}$ for all $s \in S$, where $\Galg{\HC}$ 
            denotes the split Grothendieck group of $\HC$. \textup{(}We view $\Galg{\HC}$ as a
            $\Z[v, v^{-1}]$-algebra as follows: the monoidal structure
            on $\HC$ induces a unital, associative multiplication and $v$
            acts via $v[B] \defeq [B(1)]$ for an object $B$ of $\HC$.\textup{)}
   \end{enumerate}
\end{thmlab}

It should be noted that we do not have a diagrammatic presentation of $\HC$ 
as determining the idempotents in $\BS$ is usually extremely difficult.

\subsection{The \texorpdfstring{$p$}{p}-canonical Basis}

In this section, we recall the definition of the $p$-canonical basis
and its elementary properties (see \cite{JW}). Instead of a based root
datum, we use a generalized Cartan matrix as input, but all the results
from \cite{JW} still hold in this slightly more general setting. Let $k$ 
be a field of characteristic $p \geqslant 0$. Note that the $p$-canonical basis
depends on $p$, but not on the explicit choice of $k$. Since the numerical 
input for the algorithm to calculate the $p$-canonical basis (as described 
in \cite[\S3]{JW}) reduces to the generalized Cartan matrix $A$, the $p$-canonical
basis does not depend on the choice of Kac-Moody root datum associated to $A$.

\begin{defn}
   Define $\pcan{w} = \ch([\pre{k}{B}_w])$ for all $w \in W$ where $\ch: \Galg{\HC} 
   \overset{\cong}{\longrightarrow} \heck$ is the isomorphism of $\Z[v, v^{-1}]$-algebras
   introduced earlier.
\end{defn}

We will frequently use the following elementary properties
of the $p$-canonical basis which can be found in \cite[Proposition 4.2]{JW}
unless stated otherwise:

\begin{prop}
   \label{propPCanProps}
   For all $x, y \in W$ we have:
   \begin{enumerate}
      \item $\widebar{\pcan{x}} = \pcan{x}$, i.e. $\pcan{x}$ is self-dual,
      \item $\pcan{x} = \std{x} + \sum_{y < x} \p{h}_{y, x} \std{y}$ 
            with $\p{h}_{y, x} \in \Z_{\geqslant 0}[v, v^{-1}]$,
      \item $\pcan{x} = \kl{x} + \sum_{y < x} \p{m}_{y, x} \kl{y}$
            with self-dual $\p{m}_{y, x} \in \Z_{\geqslant 0}[v, v^{-1}]$,
      \item $\iota(\pcan{x}) = \pcan{x^{-1}}$ and thus in particular 
            $\p{m}_{y, x} = \p{m}_{y^{-1}, x^{-1}}$ as well as
            $\pre{p}{h}_{y, x} = \pre{p}{h}_{y^{-1}, x^{-1}}$,
      \item $\p{m}_{y, x} = 0$ unless $\desc{L}(x) \subseteq \desc{L}(y)$
            and $\desc{R}(x) \subseteq \desc{R}(y)$,
      \item $\pcan{x} \pcan{y} = \sum_{z \in W} \p{\mu}^{z}_{x, y} \pcan{z}$
            with self-dual $\p{\mu}^{z}_{x, y} \in \Z_{\geqslant 0}[v, v^{-1}]$,
      \item $\pcan{x} = \kl{x}$ for $p = 0$ (see \cite{EW1}) and $p \gg 0$ 
            (i.e. there are only finitely many primes for which $\pcan{x} \neq \kl{x}$).
   \end{enumerate}
\end{prop}

Denote by $\mu(y, x)$ the coefficient of $v$ in the Kazhdan-Lusztig 
polynomial $\pre{0}{h}_{y, x}$.

The following result (see \cite[Lemma 4.3]{JW}) shows the remnant
of the multiplication formula for the Kazhdan-Lusztig basis. Note 
that $\pcan{s} = \kl{s}$ for all $p \geqslant 0$ as the corresponding
Schubert variety is smooth.
\begin{lem}
   \label{lemMultFormPCan}
   For $x \in W$ and $s \in \desc{L}(x)$ we have:
   \[ \kl{s} \pcan{x} = (v + v^{-1}) \pcan{x} \text{.} \]
\end{lem}
   
\section{General \texorpdfstring{$p$}{p}-Cell Theory}
\label{secpCellTheory}

\subsection{First Results}

In this section, we want to give the definition of $p$-cells. This notion
is an obvious generalization of a notion introduced by Kazhdan-Lusztig in \cite{KL}.

\begin{defn}
   For $h \in \heck$ we say that \emph{$\pcan{w}$ appears with non-zero coefficient
   in $h$} if the coefficient of $\pcan{w}$ is non-zero when expressing $h$ in
   the $p$-canonical basis.
   
   Define a preorder $\cle{R}$ (resp. $\cle{L}$) on $W$ as follows:
   $x \cle{R} y$ (resp $x \cle{L} y$) if and only if $\pcan{x}$ appears with non-zero
   coefficient in $\pcan{y} h$ (resp. $h \pcan{y}$) for some $h \in \heck$.
   Define $\cle{LR}$ to be the preorder generated by $\cle{R}$ and $\cle{L}$, in 
   other words we have: $x \cle{LR} y$ if and only if $\pcan{x}$ appears with non-zero
   coefficient in $h \pcan{y} h'$ for some $h, h' \in \heck$.
\end{defn}

Restricting $h$ and $h'$ in the definitions above to any set of generators 
of $\heck$ as a $\Z[v, v^{-1}]$-algebra, one gets a set of generating relations 
for the corresponding $p$-cell preorders (see \cite[Proposition 4.1.1]{W1}). 
The following definition introduces some notation for the relations generating 
the $p$-cell preorder obtained from the generating set 
$\{ \kl{s} \; \vert \; s \in S \}$ which we will use in \cref{secStrCells}.
\begin{defn}
   Let $x, y \in W$. We write $x \multLe{L} y$ (resp.  $x \multLe{R} y$) 
   if $\pre{p}{\mu}_{s, y}^x$ (resp. $\pre{p}{\mu}_{y, s}^x$) is non-zero 
   for some $s \in S$. In addition, we write $x \multLe{2} y$ if
   $x \multLe{L} y$ or $x \multLe{R} y$ holds.
\end{defn}

For the sake of completeness, we will state explicitly that these elementary relations
generate the $p$-cell-preorder (see \cite[Lemma 5.3]{AHRHumphreysConj} for a proof):

\begin{lem}
   \label{lemGenCells}
   For $x, y \in W$ the following holds:
   $x \cle{R} y$ if and only if there exists a chain $x = x_0 \multLe{R} x_1 
   \multLe{R} \dots \multLe{R} x_k = y$. Similarly for the left (resp. two-sided) 
   $p$-cell preorder.
\end{lem}

In the remainder of the section, we will prove some elementary
properties of $p$-cells. In most cases we will focus on right $p$-cells
and not state the version for left $p$-cells explicitly.

In \cite[Proposition 2.4]{KL} Kazhdan-Lusztig observed that a Kazhdan-Lusztig right 
cell preorder relation implies an inclusion of left descent sets. The following 
result shows that the compatibility between cells and descent sets carries over 
to the more general setting. More precisely, the set of all elements with a fixed 
left descent set is a union of right $p$-cells. The result can also be found
in \cite[Lemma 5.4]{AHRHumphreysConj}:

\begin{lem}
   \label{lemCleDesc}
   For $x, y \in W$ with $y \cle{R} x$ we have $\desc{L}(x) \subseteq \desc{L}(y)$.
   In particular, $x \ceq{R} y$ gives $\desc{L}(x) = \desc{L}(y)$ and for any
   $I \subseteq S$ the set $\{w \in W \; \vert \; \desc{L}(w) = I \}$ is a union
   of right $p$-cells.
\end{lem}
\begin{proof}
   Due to \cref{lemGenCells} it is enough to consider the case 
   where we multiply $\pcan{x}$ with $\pcan{s}$ for $s \notin \desc{R}(x)$. 
   We have on the one hand:
   \[ \pcan{x} \pcan{s} = \sum_y \p{\mu}_{x,s}^{y} \pcan{y} \]
   On the other hand we can write:
   \begin{align*} 
   \pcan{x} \pcan{s} &= (\sum_{y \leqslant x} \p{m}_{y,x} \kl{y}) \kl{s} \\
                     &= \sum_{\substack{y \leqslant x \\ s \in \desc{R}(y)}}(v+v^{-1})
                        \p{m}_{y, x}\kl{y} + \sum_{\substack{y \leqslant x \\ s \not\in \desc{R}(y)}}
                        \p{m}_{y, x}(\kl{ys} + \sum_{\substack{z \leqslant y \\ s \in \desc{R}(z)}}
                        \mu(z, y) \kl{z})
   \end{align*}
   \Cref{propPCanProps}(v) shows that all $y \in W$ occurring with non-zero
   $\p{m}_{y,x}$ on the right hand side satisfy $\desc{L}(x) \subseteq \desc{L}(y)$.
   \cite[Proposition 2.4]{KL} shows that $z \cle[0]{R} y$ implies $\desc{L}(y) \subseteq \desc{L}(z)$.
   Observe that the set of $y \in W$ with non-zero structure coefficient 
   $\p{\mu}_{x,s}^{y}$ is a subset of the set of all $y \in W$ indexing 
   a summand $\kl{y}$ with non-zero coefficient on the right hand side 
   (due to \cref{propPCanProps}(iii) and (vi)). Putting all of this together 
   gives the result.
\end{proof}

 
\begin{cor}
   \label{corIdCell}
   $\{\id\}$ is a left, right, and $2$-sided $p$-cell for all primes $p$.
\end{cor}

It is well known for Kazhdan-Lusztig cells that left and right cells are 
closely related via taking inverses. Using the $\Z[v, v^{-1}]$-linear anti-involution 
$\iota$ on $\heck$ together with \cref{propPCanProps}(iv) we obtain the
corresponding result for $p$-cells which will allow us to pass from left
to right $p$-cells:

\begin{lem}
   \label{corLeftRightEq}
   For all $x, y \in W$ we have:
   \begin{align*}
      x \cle{L} y \; &\Longleftrightarrow \; x^{-1} \cle{R} y^{-1} \text{,}\\
      x \cle{LR} y \; &\Longleftrightarrow \; x^{-1} \cle{LR} y^{-1} \text{.}
   \end{align*}
\end{lem}

Next, we want to consider the question which automorphisms of our Coxeter system
induce automorphisms on $\heck$ that are well-behaved with respect to the $p$-canonical basis.
Let $\phi: (W, S) \overset{\sim}{\longrightarrow} (W, S)$ be an automorphism 
of Coxeter systems (in particular we have $\phi(S) = S$) which leaves the generalized
Cartan matrix $A$ invariant when permuting simultaneously the corresponding rows and columns 
(i.e. $\langle \cort[t], \rt[s] \rangle = \langle \cort[\phi(t)], 
\rt[\phi(s)] \rangle$ for all $s, t \in S$). Then $\phi$ induces a 
$\Z[v, v^{-1}]$-linear automorphism of $\heck$ via $\std{x} \mapsto \std{\phi(x)}$
for $x \in W$ which we will also denote by $\phi$ by slight abuse of 
notation. Therefore, $\phi$ maps $\kl{x}$ to $\kl{\phi(x)}$ by the defining
property of the Kazhdan-Lusztig basis.

\begin{prop}
   \label{propCoxAutom}
   In the setting given above we have for all $x, y \in W$:
   \begin{enumerate}
      \item $\phi(\pcan{x}) = \pcan{\phi(x)}$,
      \item $\pre{p}{m}_{y, x} = \pre{p}{m}_{\phi(y), \phi(x)}$ and 
            $\pre{p}{h}_{x,y} = \pre{p}{h}_{\phi(x), \phi(y)}$,
      \item $\pre{p}{\mu}_{x,y}^z = \pre{p}{\mu}_{\phi(x), \phi(y)}^{\phi(z)}$, 
      \item $x \cle{L} y \Leftrightarrow \phi(x) \cle{L} \phi(y)$ and 
            $x \cle{R} y \Leftrightarrow \phi(x) \cle{R} \phi(y)$.
   \end{enumerate}
\end{prop}
\begin{proof}
   Observe that $\phi$ induces a monoidal, $k$-linear equivalence 
   of $\BS$ and thus of $\HC$ which on the $\Hom$-spaces merely permutes the 
   colours in the diagrams (given by $S$) and the variables of the polynomials 
   in $R$ decorating the regions according to the action of $\phi$. Since the 
   numerical input for the algorithm to calculate the $p$-canonical basis (as 
   described in \cite[\S3]{JW}) reduces to $A$ we see immediately that this 
   equivalence sends $\pre{k}{B}_x$ to $\pre{k}{B}_{\phi(x)}$ and thus
   on the level of Grothendieck groups $\pcan{x}$ to $\pcan{\phi(x)}$. This proves (i).
   
   Recall that $\phi$ maps $\kl{x}$ to $\kl{\phi(x)}$ and $\std{x}$ to $\std{\phi(x)}$ for all $x \in W$.
   For this reason, (ii) follows from (i) by rewriting  $\pcan{x}$ in the Kazhdan-Lusztig basis (resp. standard
   basis), applying $\phi$, using (i) and comparing coefficients in the Kazhdan-Lusztig 
   basis (resp. standard basis). (i) implies (iii) in a similar way and (iv) follows
   from (iii).
\end{proof}

Suppose that our based root datum is irreducible. In this case, the last proposition 
can be applied to all automorphisms of the (extended) Dynkin diagram of our root 
system. In finite type conjugation by the longest element in the finite Weyl group 
is also covered by the last proposition. 
Indeed, it follows from \cite[Remark 13.1.8]{DaCoxGrps} that for irreducible finite Coxeter 
groups the longest element $w_0$ is central except in types $A_n$ for 
$n \geqslant 2$, $D_n$ with $n$ odd, $E_6$, and $I_2(m)$ for $m$ odd where $I_2(m)$ denotes
the dihedral group of order $2m$. In all these cases, conjugation by $w_0$ gives 
the obvious automorphism of the corresponding Coxeter graph. After restricting to 
crystallographic Coxeter systems, only simply-laced types remain and so any automorphism
of the Coxeter graph gives an automorphism of the Dynkin diagram of the same type in
the obvious way (as the graphs are isomorphic).

\begin{defn}
   \label{defFinSubset}
   Let $I \subseteq S$ be a subset.
   Call $I$ \emph{finitary} if the corresponding parabolic subgroup 
   $\langle I \rangle \subseteq W$ is finite. Define $W^I$ to be the set of 
   representatives of minimal length of cosets in $W / W_I$.
\end{defn}

The following result is the main result of this section and generalizes the parabolic 
compatibility for Kazhdan-Lusztig cells (see \cite[Proposition 9.11]{LuUneq}) to the 
setting of $p$-cells:

\begin{thmlab}[Parabolic compatibility of right $p$-cells]
   \label{thmParaCombPCells}
   Let $I \subseteq S$ be a finitary subset. Then for $y, z \in W_I$ the following holds:
   \[ z \cle{R} y \text{ in } W_I \Leftrightarrow \forall x \in W^I: xz \cle{R} xy \text{ in } W \]
\end{thmlab}

%

As a corollary to the proof of \Cref{thmParaCombPCells} we get:

\begin{cor}
   \label{corParPKLPols}
   In the setting of \Cref{thmParaCombPCells} we have:
   \[ \pre{p}{h}_{xy, xz} = \pre{p}{h}_{y,z} \]
\end{cor}

\begin{remark}
   Let $H \subseteq W$ be a finite standard parabolic subgroup and $C \subseteq W$
   a right $p$-cell. \cref{thmParaCombPCells} implies immediately that 
   the cell module associated to $C$ when restricted to $H$ decomposes
   into cell modules for $H$ (see \cite[Proposition 3.11]{BVPrimIdealsExcepGrps} and
   \cite[Proposition 1]{RIndResCells}).
\end{remark}

\subsection{Proof of Theorem \ref{thmParaCombPCells} and Corollary \ref{corParPKLPols}}

First we will prove \Cref{thmParaCombPCells}: For all elements $w$ in $W_I \cup W^I$ 
choose a reduced expression $\expr{w}$. We have a bijection
\begin{align*}
   W^I \times W_I &\longrightarrow W \\
    (x, y) &\longmapsto xy
\end{align*}
such that $l(xy) = l(x) + l(y)$ (see \cite[Proposition 2.4.4]{BBCoxGrps}).
Therefore, for $x \in W^I$ and $y \in W_I$ the concatenation of the corresponding
reduced expressions $\expr{x}$ and $\expr{y}$ gives a reduced expression $\expr{x}^{\frown}\expr{y}$
of $xy$. Choose $x \in W^I$ and $y \in W_I$ arbitrarily. In the following we
sometimes identify a subexpression with its associated $01$-sequence.

\begin{lem}
   \label{lemBij}
   For all $z \in W_I$ we have a decoration- and defect-preserving bijection:
   \begin{alignat*}{3}
      g: \{ \text{subexpr.  of } &\expr{y} \text{ expressing } z\} \ 
         &&\overset{\sim}{\longrightarrow} \ \{ \text{subexpr.} &&\text{ of } \expr{x} \expr{y} \text{ expressing } xz \} \\
      &\expr{e} &&\longmapsto &&\underbrace{(1,\dots, 1)}_{l(x)\text{ ones}}{}^{\frown}\expr{e}
   \end{alignat*}
\end{lem}
\begin{proof}
   It is easy to see that $g$ is well-defined and injective.
   Leaving out any letter of $\expr{x}$ in a subexpression of $\expr{x}\expr{y}$
   leads to a Bruhat stroll ending in a right coset $\widetilde{x} W_I$ with 
   $\widetilde{x} < x$ as $x$ is of minimal length in $xW_I$. Thus, $g$ is
   surjective.
   
   For the definition of the Bruhat graph we refer the reader to \cite[Definition
   1.1]{DBruhatGraph}. The induced Bruhat graph on $x W_I$ (as a subgraph of the Bruhat graph 
   of $W$) is isomorphic to the one of $W_I$ (see \cite[Theorem 1.4]{DBruhatGraph}). 
   This follows for example from the subword property (see \cite[Theorem 2.2.2]{BBCoxGrps}). 
   Thus any subexpression $g(\expr{e})$ of $\expr{x} \expr{y}$ expressing $xz$ 
   will have a decoration starting with $l(x)$ symbols $U1$ (as $\expr{x}$ is 
   a reduced expression) and the remaining expression $\expr{e}$ will be 
   decorated in the same way as $\expr{e}$ would be decorated as a subexpression 
   of $\expr{y}$ expressing $z$. Since the ones in a subexpression do
   not contribute to the defect, this immediately implies:
   \[ \df (g(\expr{e})) = \df (\expr{e}) \]
   where on the left (resp. right) hand side the defect is calculated as a subexpression
   of $\expr{x} \expr{y}$ (resp. $\expr{y}$).
\end{proof}

This bijection matches up the combinatorial data used to define the light leaves 
and thus allows us to compare the corresponding local intersection forms.
Consider the local intersection forms $I_{\expr{x}\expr{y}, xz}$ of $\expr{x} \expr{y}$
at $xz$ (resp. $I_{\expr{y}, z}$ of $\expr{y}$ at $z$) and the matrices 
representing them with respect to the light leaves bases (see \cite[\S3]{JW}).
For two subexpressions $\expr{e}, \expr{f}$ of $\expr{y}$ expressing $z$ 
we get in $\pre{k}{\HCnl{xz}} \otimes_R k$:
\[ (I_{\expr{x}\expr{y}, xz})_{g(\expr{e}), g(\expr{f})}  = 
      \id_{\BS(\expr{x})} (I_{\expr{y}, z})_{\expr{e}, \expr{f}} \]
This implies that the multiplicity of $\pre{k}{B}_{xz}$ and its grading shifts
in $\BS(\expr{x} \expr{y})$ which is given by $\grk(I_{\expr{x}\expr{y}, xz})$
coincides with the multiplicity of $\pre{k}{B}_z$ and its grading shifts
in $\BS(\expr{y})$ which is given by $\grk(I_{\expr{y}, z})$.

Choose any total order on $W$ refining the Bruhat order and preserving
elements in the same coset in $W / W_I$ as blocks of adjacent
elements. Note that our choices above have fixed a reduced expression $\expr{w}$
for any element $w \in W$. Denote by $M$ the base change matrix
from the Bott-Samelson basis $\{ \kl{\expr{w}} \; \vert \; w \in W_I \}$
to the $p$-canonical basis $\{ \pcan{w} \; \vert \; w \in W_I\}$ of $\heck[(W_I, I)]$.
$M$ is an upper-triangular, invertible matrix with entries in $\Z[v, v^{-1}]$  and
ones on the diagonal. The above considerations show that the base change matrix 
from the Bott-Samelson to the $p$-canonical basis of $\heck$ looks as follows:

\[ \bordermatrix{ & W_I & \dots & xW_I & \dots  \cr
   \phantom{x}W_I & M &  &  &  \cr 
   \phantom{X}\vdots &  & \ddots & & \raisebox{2ex}[0pt]{\BigAst} \cr
   xW_I &  &  & M & \cr
   \phantom{X}\vdots & \multicolumn{2}{c}{\raisebox{2ex}[0pt]{\BigZero}} & & \ddots & }
\]

Note that this form is preserved by taking inverses. This means that when expressing 
$\pcan{xy}$ in terms of the Bott-Samelson basis $M^{-1}$ gives all coefficients for terms
indexed by $\expr{x} \expr{z}$ for $z \in W_I$. Using this allows us to partly
decouple the terms $\kl{\expr{x}}$ and $\kl{\expr{y}}$. When calculating $\pcan{xy} \pcan{w}$
for $w \in W_I$ we can simply express both terms in the Bott-Samelson basis,
perform the calculation where only the structure coefficients for the Bott-Samelson
basis of $\heck[(W_I, I)]$ come into play and rewrite it in terms of the $p$-canonical basis.
This immediately implies the following result as we have full control over the
situation in the top coset $x W_I$:

\begin{cor}
   \label{corParStrCoeffs}
   For all $y, z, w \in W_I$ and a minimal coset representative $x$ of $W / W_I$
   we have:
   \[ \pre{p}{\mu}^{xz}_{xy, w} = \pre{p}{\mu}^z_{y,w} \]
\end{cor}

The last corollary proves \Cref{thmParaCombPCells} since elementary relations obtained from
$\pcan{w} \pcan{s}$ for $w \in W_I$ and $s \in I$ generate the right $p$-cell
preorder in $W_I$ (see \cref{lemGenCells}).

\noindent Next, we will prove \Cref{corParPKLPols}:

Recall the following lemma which in the case of a reduced expression $\expr{w}$ 
describes how to express the Bott-Samelson basis element $\kl{\expr{w}}$
in terms of the standard basis (see \cite[Lemma 2.10]{EW2}):

\begin{lem}
   For any expression $\expr{w}$ in $S$ we have:
   \[ \kl{\expr{w}} = \sum_{\substack{\expr{e}  \text{ subexpression} \\ \text{of } \expr{w}}} 
      v^{\df(\expr{e})} \std{(\expr{w}^{\expr{e}})_{\bullet}} \]
\end{lem}

Recall that we have chosen a total order on $W$ in the proof of \Cref{thmParaCombPCells}.
Denote by $B$ the base change matrix from the Bott-Samelson basis 
$\{ \kl{\expr{w}} \; \vert \; w \in W_I \}$ to the standard basis
$\{ \std{w} \; \vert \; w \in W_I\}$ of $\heck[(W_I, I)]$. Then $B$ is an 
upper-triangular, invertible matrix with entries in $\Z[v, v^{-1}]$  and  
ones on the diagonal. The defect-preserving bijection from \cref{lemBij} 
shows that the base change matrix from the Bott-Samelson to the standard 
basis of $\heck$ looks as follows:

\[ \bordermatrix{ & W_I & \dots & xW_I & \dots  \cr
   \phantom{x}W_I & B &  &  &  \cr 
   \phantom{X}\vdots &  & \ddots & & \raisebox{2ex}[0pt]{\BigAst} \cr
   xW_I &  &  & B & \cr
   \phantom{X}\vdots & \multicolumn{2}{c}{\raisebox{2ex}[0pt]{\BigZero}} & & \ddots & }
\]

Multiplying the base change matrix from the $p$-canonical to the Bott-Samelson
basis with the base change matrix from the Bott-Samelson to the standard basis 
finishes the proof of \Cref{corParPKLPols}.

\subsection{Decomposition Criterion for Kazhdan-Lusztig Cells}
\label{secDecompCriterion}

In this section, we want to study the interplay between the weak right Bruhat order 
(see \cite[Definition 3.1.1]{BBCoxGrps} for the definition) and the right 
$p$-cell preorder. This will allow us to formulate a simple criterion as to when
$p$-cells decompose into Kazhdan-Lusztig cells.

In the next few results we will focus on right cells, but a similar version
for left cells can easily be formulated. Throughout this section, we assume
that $W$ has finitely many right Kazhdan-Lusztig cells. This is known for
finite and affine Weyl groups by \cite[Theorem 2.2 (a)]{LuCellsInAffWeylGrpsII},
but there are examples of crystallographic Coxeter groups with infinitely
many right Kazhdan-Lusztig cells (see \cite{BeHypCoxGrps, BeRightAngledCoxGrps}).

For $x \in W$, let $\expr{x} = (s_1, s_2, \dots, s_k)$ be a reduced expression.
Set $x_i \defeq s_1 s_2 \dots s_i$ for all $0 \leqslant i \leqslant k$.
Since $\pcan{e} = \kl{e}$ there exists a maximal $0 \leqslant m \leqslant k$
such that for all $y \leqslant x_m$ with $\pre{p}{m}_{y, x_m} \ne 0$ we have
$y \cle[0]{R} x_m$. In this setting we have the following result:

\begin{lem}
   All $y \leqslant x$ with $\pre{p}{m}_{y, x}$ non-zero satisfy: $y \cle[0]{R} x_m$.
\end{lem}
\begin{proof}
   The claim follows as $\pcan{x_{l}}$ for $l \geqslant m$ is a linear combination of 
   Kazhdan-Lusztig basis elements indexed by elements in $\{ w \in W \; \vert \; 
   w \cle[0]{R} x_m \}$.
\end{proof}

Observe that $x_m \cge[0]{R} x$ always holds. So if $x_m \cle[0]{R} x$, then $x$
and $x_m$ lie in the same Kazhdan-Lusztig right cell.

\begin{cor}
   \label{corFirstStepDecompCrit}
   If $x_m$ and $x$ lie in the same Kazhdan-Lusztig right cell, then for 
   $y \leqslant x$ with $\pre{p}{m}_{y, x}$ non-zero we have $y \cle[0]{R} x$.
\end{cor}

\begin{defn}
   Let $C \subseteq W$ be an arbitrary subset. $C$ is called \emph{right-connected}
   if for every two elements $x, y \in C$, there exists a sequence $x = x_0$, 
   $x_1$, $\dots$, $x_k = y$ in $C$ such that $x_i^{-1} x_{i-1} \in S$ (i.e. $x_{i-1}$ 
   and $x_i$ differ by a simple reflection on the right) for all $1 \leqslant i \leqslant k$. 
   It follows that $C$ decomposes as a disjoint union of its right-connected components, 
   i.e. the maximal right-connected subsets.
   
   Call an element $x \in C$ \emph{right-minimal} if $x$ cannot be reached from
   any other element $y \in C\setminus \{x\}$ via a sequence $y= x_0$, $x_1$, 
   $\dots$, $x_k = x$ in $C$ as above satisfying in addition $y < x_1 < x_2 < 
   \dots < x_k = x$. Observe that an element is right-minimal if
   and only if it is minimal with respect to the weak right Bruhat order.
   
   Similarly we define \emph{left-connected} and \emph{left-minimal} using
   multiplication by simple reflections on the left, as well as \emph{$2$-connected}
   and \emph{$2$-minimal}.
\end{defn}

The following observation follows immediately from the definition of a right-minimal
element, but shows their most important property:
\begin{lem}
   \label{lemRMinElts}
   Let $C \subseteq W$ be an arbitrary subset. For all $y \in C$ there exists
   a right-minimal element $x \in C$ such that $y \cle[0]{R} x$ and $y \cle{R} x$.
\end{lem}

At this point we should mention the following conjecture by Lusztig which he originally
formulated for (affine) Weyl groups in \cite[p. 14]{HKOpenProbs}. It still appears to 
be open in finite type $D_n$ and in general affine type:
\begin{conj}
   Every Kazhdan-Lusztig right cell in a Coxeter group is right-connected.
\end{conj}
In finite type, the conjecture is known to hold for all dihedral groups,
in type $A_n$ (see \cite[\S5]{KL}), $B_n$ (see \cite[Theorem 3.5.9]{GaPrimIdealsIII})
and in all exceptional types $H_3$, $H_4$, $F_4$, $E_6$, $E_7$ and $E_8$ 
(see \cite[Example 7.3]{GHCellsInE8}). In affine type, it has been verified in 
affine rank $2$ (see \cite[Theorem 11.3]{LuCellsInAffWeylGrpsI}), 
for $W$ of type $\widetilde{A}_n$ with $n \geqslant 1$ (see \cite[Theorem 18.2.1]{ShCellsInAffA}), 
for Kazhdan-Lusztig right cells contained in the lowest Kazhdan-Lusztig two-sided 
cell (see \cite[Corollary 1.2]{ShTwoCellII}), for Kazhdan-Lusztig right cells contained 
in the Kazhdan-Lusztig two-sided cell of elements with a unique reduced expression (see \cite[Proposition 3.8]{Lu1}) 
and some other special cases (see for example \cite{XLeftConnected}, \cite[Theorem 4.8]{ShCoxEltsCells} and 
\cite{ShFullyCommEltsCells}).

In the rest of the section we want to apply these notions to compare Kazhdan-Lusztig
and $p$-cells by looking at minimal elements. We will focus on right cells even
though there are similar results about left (resp. two-sided) cells. Using
a similar idea as in \Cref{corFirstStepDecompCrit} we obtain the following result:
\begin{cor}
   \label{corCoeff}
   Let $C$ be a Kazhdan-Lusztig right cell and $C_{\text{min}}$ the set of 
   right-minimal elements in $C$. Assume for all $x \in C_{\text{min}}$ and 
   $y \leqslant x$ the following:
   \[ \pre{p}{m}_{y, x} \ne 0 \Rightarrow y \cle[0]{R} x \]
   Then for all $x \in C$ and $y \leqslant x$ with $\pre{p}{m}_{y,x} \ne 0$ we have 
   $y \cle[0]{R} x$.
\end{cor}

\begin{defn}
   Let $X$ be a set equipped with a preorder $\leqslant$. A subset $Y \subseteq X$
   is called a \emph{lower set} if for $y \in Y$ and any $x \in X$ with $x \leqslant y$ 
   we have $x \in Y$ as well.
\end{defn}

Observe that any lower set in the right $p$-cell preorder can be written as a union of 
right $p$-cells. For this reason the following result is the starting point of
our criterion:

\begin{lem}
   \label{lemUnion}
   Let $C$ be a Kazhdan-Lusztig right cell that satisfies the assumptions of
   \cref{corCoeff} and $x \in C$. Then $y \cle[p]{R} x$ implies $y \cle[0]{R} x$. 
   If $C$ is minimal in the Kazhdan-Lusztig right cell preorder, then $C$ is a 
   lower set in the right $p$-cell preorder and we have
   \[ \bigoplus_{x \in C} \Z[v, v^{-1}] \pcan{x} = \bigoplus_{x \in C} \Z[v, v^{-1}] \kl{x} \]
   as $\Z[v, v^{-1}]$-submodules of $\heck$.
\end{lem}
\begin{proof}
   Assume that $\pcan{y}$ occurs with non-zero coefficient in $\pcan{x} h$ for some 
   $h \in \heck$. Write $\pcan{x} h = (\kl{x} + \sum_{z < x} \pre{p}{m}_{z, x} \kl{z}) h$.
   If $\kl{y}$ occurs with non-zero coefficient in the product $\kl{x} h$ then 
   we have by definition $y \cle[0]{R} x$. If $\kl{y}$ occurs with non-zero 
   coefficient in one of the products $\pre{p}{m}_{z, x} \kl{z} h$, then
   we have $\pre{p}{m}_{z, x} \neq 0$ and by \cref{corCoeff} that $y \cle[0]{R} z \cle[0]{R} x$.
   
   Denote by $C_{\text{min}}$ the set of right-minimal elements in $C$. Our arguments above
   show that we have the following inclusions:
   \[ \{ \cle{R} C \} = \bigcup_{x \in C_{\text{min}}} \{ \cle{R} x\} \subseteq 
      \bigcup_{x \in C_{\text{min}}} \{ \cle[0]{R} x\} = \{ \cle[0]{R} C \} \]
   The equalities on the left and right hand side follow from \cref{lemRMinElts}. 
   In particular, $C$ is contained in the left hand side. Thus, if $C$ 
   is minimal in the Kazhdan-Lusztig right cell preorder, we actually have equality which 
   implies the claim as the left hand side obviously is a lower set in the $p$-cell
   preorder. In this case, \cref{corCoeff} shows that for any $x \in C$ the $p$-canonical basis
   element $\pcan{x}$ can be written in terms of Kazhdan-Lusztig basis elements indexed
   by elements in $C$ which implies the statement about the span of the $p$-canonical
   and the Kazhdan-Lusztig basis elements.
\end{proof}

\begin{lem}
   Let $C$ be a Kazhdan-Lusztig right cell. Assume that all Kazhdan-Lusztig right cells smaller or
   equal than $C$ in the Kazhdan-Lusztig right cell preorder satisfy the assumption of \cref{corCoeff}.
   Then $\{ \cle[0]{R} C \}$ is a lower set in the right $p$-cell preorder and we have
   \[ \bigoplus_{x \in \{ \cle{R} C \}} \Z[v, v^{-1}] \pcan{x} = 
      \bigoplus_{x \in \{ \cle[0]{R} C \}} \Z[v, v^{-1}] \kl{x} \]
   as $\Z[v, v^{-1}]$-submodules of $\heck$. Moreover, $C$ decomposes as a union
   of right $p$-cells.
\end{lem}
\begin{proof}
   We proceed by induction on the height of $C$ in the Kazhdan-Lusztig right cell preorder.
   By assumption the height of any cell in the right cell preorder is finite.
   
   \cref{lemUnion} gives the induction start. Let $C$ be of height $\geqslant 2$.
   By induction for all predecessors of $C$ in the Kazhdan-Lusztig right cell preorder we know
   our claim holds for $\{ \clt[0]{R} C \}$. Therefore, we may pass to the quotient
   \[\heck / \bigoplus_{x \in \{ \clt[0]{R} C \}} \Z[v, v^{-1}] \kl{x}\]
   where $C$ becomes the smallest cell. Note that this quotient is a right $\heck$-module
   that is free as a $\Z[v, v^{-1}]$-module and admits a $p$-canonical as well
   as a Kazhdan-Lusztig basis. Denote by
   \[ \pi: \heck \rightarrow \heck / \bigoplus_{x \in \{ \clt[0]{R} C \}} \Z[v, v^{-1}] \kl{x}\]
   the projection to the quotient. Applying $\pi$ amounts to forgetting all basis elements
   indexed by elements in $\{ \clt[0]{R} C \}$ when expressing an element in the 
   Hecke algebra in the $p$-canonical or Kazhdan-Lusztig basis. For this reason,
   the action of $\heck$ on this right module captures a lot of information
   about the $p$-cell as well as the Kazhdan-Lusztig cell structure in the following sense:
   For $x, y \in \{ \underset{R}{\overset{0}{\not <}} C \}$ we have $y \cle{R} x$ if and only if
   there exists an element $h \in \heck$ such that $\pi(\pcan{y})$ occurs 
   with non-trivial coefficient in $\pi(\pcan{x})h$. A similar statement holds 
   for the Kazhdan-Lusztig cell structure. This allows us to conclude as in the 
   proof of \cref{lemUnion}.
\end{proof}

\subsection{(Counter-)Examples}
\label{secCellEx}

In this section, we will give some examples of the $p$-cell structure of
finite Weyl groups. In this section we will only give the Dynkin 
diagram and consider the corresponding Cartan matrix as input.
One may obtain a Kac-Moody root datum from any based root datum
of the corresponding connected semi-simple algebraic group.

We will restrict to those examples that give counterexamples to obvious generalizations
of known results from Kazhdan-Lusztig cell theory. All results in this
section were obtained using computer calculations. Denote by $w_0$ the longest
element in the corresponding finite Weyl group.

\subsubsection{Type \texorpdfstring{$B_2$}{B2}}
\label{secB2}

We label the simple reflections as follows:
\[\begin{tikzpicture}[auto, baseline=(current  bounding  box.center)]
      \draw (0,\edgeShift) -- (-1,\edgeShift);
      \draw (0,-\edgeShift) -- (-1,-\edgeShift);
      \path (0,0) to node[Greater] (mid) {} (-1,0);
      \draw (mid.center) to +(30:\wingLen);
      \draw (mid.center) to +(330:\wingLen);
      \node [DynNode] (1) at (-1,0) {$s$};
      \node [DynNode] (2) at (0,0) {$t$}; 
\end{tikzpicture} \]

The following diagrams show the right (resp. two-sided) cells and the 
corresponding preorders in type $B_2$:

\begin{center}
   \begin{longtable}{l | c | c }
      & right cells & two-sided cells \\ \hline \Top
      
      KL-cells: &
      \begin{tikzpicture}[scale=0.75, auto, baseline=(current  bounding  box.center)]
         \node (id) at (0,0) {$\{ \id \}$};
         \node (s) at (-2, -1.5) {\small$\{ s, st, sts \}$};
         \node (t) at (2, -1.5) {\small$\{ t, ts, tst \}$};
         \node (w0) at (0, -3) {\small$\{ w_0 \}$};

         \draw (id) to (s);
         \draw (id) to (t);
         \draw (t) to (w0);
         \draw (s) to (w0);
      \end{tikzpicture} &
      \begin{tikzpicture}[scale=0.75, auto, baseline=(current  bounding  box.center)]
         \node (id) at (0,0) {\small$\{ \id \}$};
         \node (C) at (0, -1.5) {\small$\{ s, t, st, ts, sts, tst \}$};
         \node (w0) at (0, -3) {\small$\{ w_0 \}$};

         \draw (id) to (C);
         \draw (C) to (w0);
      \end{tikzpicture} \\ \hline
      $2$-cells: & 
      \begin{tikzpicture}[scale=0.75, auto, baseline=(current  bounding  box.center)]
         \node (id) at (0,0) {\small$\{ \id \}$};
         \node (s) at (-2, -1.5) {\small$\{ s \}$};
         \node (sothers) at (-2, -3) {\small$\{ st, sts \}$};
         \node (t) at (2, -2.25) {\small$\{ t, ts, tst \}$};
         \node (w0) at (0, -4.5) {\small$\{ w_0 \}$};

         \draw (id) to (s);
         \draw (id) to (t);
         \draw (t) to (w0);
         \draw (s) to (sothers);
         \draw (sothers) to (w0);
      \end{tikzpicture} &
      \begin{tikzpicture}[scale=0.75, auto, baseline=(current  bounding  box.center)]
         \node (id) at (0,0) {\small$\{ \id \}$};
         \node (s) at (0, -1.5) {\small$\{ s \}$};
         \node (C) at (0, -3) {\small$\{ t, st, ts, sts, tst \}$};
         \node (w0) at (0, -4.5) {\small$\{ w_0 \}$};

         \draw (id) to (s);
         \draw (s) to (C);
         \draw (C) to (w0);
      \end{tikzpicture} \\
   \end{longtable}
\end{center}

Lusztig showed in \cite[Proposition 3.8]{Lu1} that the set $C$ of non-trivial elements
in a Coxeter group that have a unique reduced expression always forms a Kazhdan-Lusztig two-sided
cell and for any $s \in S$ the set ${}_s C \defeq \{ w \in C \; \vert \; \desc{L}(w) = \{s\}\}$
gives a Kazhdan-Lusztig right cell. The example above shows that both statements
do not hold for $p$-cells in general. 

Observe that in characteristic $0$ we have for $x, y \in W$ 
(see \cite[Corollary 11.7]{LuUneq}) the following equivalences
\[ x \cle[0]{R} y \Leftrightarrow y w_0 \cle[0]{R} x w_0 \Leftrightarrow w_0 y \cle[0]{R} w_0 x \]
and the same statement for the left and two-sided cell preorder. 
The example above also shows that the analogous statement does not hold for 
$p$-cells. 

\subsubsection{Type \texorpdfstring{$G_2$}{G2}}

We label the simple reflections as follows:
\[\begin{tikzpicture}[auto, baseline=(current  bounding  box.center)]
      \draw (0,2*\edgeShift) -- (-1,2*\edgeShift);
      \draw (0,0) -- (-1,0);
      \draw (0,-2*\edgeShift) -- (-1,-2*\edgeShift);
      \path (0,0) to node[Greater] (mid) {} (-1,0);
      \draw (mid.center) to +(30:\wingLen);
      \draw (mid.center) to +(330:\wingLen);
      \node [DynNode] (1) at (-1,0) {$s$};
      \node [DynNode] (2) at (0,0) {$t$};
   \end{tikzpicture} \]

The following diagrams show the right (resp. two-sided) cells and the 
corresponding preorders in type $G_2$ using the notation from the last
subsection. In particular, $C$, ${}_s C$ and ${}_t C$ are defined as in 
\cref{secB2}.

\begin{center}
   \begin{longtable}{l | c | c }
      & right cells & two-sided cells \\ \hline \Top
      
      KL-cells: &
      \begin{tikzpicture}[scale=0.75, auto, baseline=(current  bounding  box.center)]
         \node (id) at (0,0) {$\{ \id \}$};
         \node (s) at (-2, -1.5) {\small${}_s C$};
         \node (t) at (2, -1.5) {\small${}_t C$};
         \node (w0) at (0, -3) {\small$\{ w_0 \}$};

         \draw (id) to (s);
         \draw (id) to (t);
         \draw (t) to (w0);
         \draw (s) to (w0);
      \end{tikzpicture} &
      \begin{tikzpicture}[scale=0.75, auto, baseline=(current  bounding  box.center)]
         \node (id) at (0,0) {\small$\{ \id \}$};
         \node (C) at (0, -1.5) {\small$C$};
         \node (w0) at (0, -3) {\small$\{ w_0 \}$};

         \draw (id) to (C);
         \draw (C) to (w0);
      \end{tikzpicture} \\ \hline
      $2$-cells: & 
      \begin{tikzpicture}[scale=0.75, auto, baseline=(current  bounding  box.center)]
         \node (id) at (0,0) {\small$\{ \id \}$};
         \node (s) at (-2, -1.5) {\small$\{ s, st \}$};
         \node (sothers) at (-2, -3) {\small${}_s C \setminus \{ s, st \}$};
         \node (t) at (2, -1.5) {\small$\{ t, ts \}$};
         \node (tothers) at (2, -3) {\small${}_t C \setminus \{t, ts\}$};
         \node (w0) at (0, -4.5) {\small$\{ w_0 \}$};

         \draw (id) to (s);
         \draw (id) to (t);
         \draw (s) to (sothers);
         \draw (sothers) to (w0);
         \draw (t) to (tothers);
         \draw (tothers) to  (w0);
      \end{tikzpicture} &
      \begin{tikzpicture}[scale=0.75, auto, baseline=(current  bounding  box.center)]
         \node (id) at (0,0) {\small$\{ \id \}$};
         \node (s) at (0, -1.5) {\small$\{ s, t, st, ts \}$};
         \node (C) at (0, -3) {\small$C \setminus \{  s, t, st, ts \}$};
         \node (w0) at (0, -4.5) {\small$\{ w_0 \}$};

         \draw (id) to (s);
         \draw (s) to (C);
         \draw (C) to (w0);
      \end{tikzpicture} \\ \hline
      $3$-cells: & 
      \begin{tikzpicture}[scale=0.75, auto, baseline=(current  bounding  box.center)]
         \node (id) at (0,0) {\small$\{ \id \}$};
         \node (s) at (-2, -1.5) {\small$\{ s \}$};
         \node (sothers) at (-2, -3) {\small${}_s C \setminus \{ s \}$};
         \node (t) at (2, -2.25) {\small${}_t C$};
         \node (w0) at (0, -4.5) {\small$\{ w_0 \}$};

         \draw (id) to (s);
         \draw (id) to (t);
         \draw (t) to (w0);
         \draw (s) to (sothers);
         \draw (sothers) to (w0);
      \end{tikzpicture} &
      \begin{tikzpicture}[scale=0.75, auto, baseline=(current  bounding  box.center)]
         \node (id) at (0,0) {\small$\{ \id \}$};
         \node (s) at (0, -1.5) {\small$\{ s \}$};
         \node (C) at (0, -3) {\small$C \setminus \{ s \}$};
         \node (w0) at (0, -4.5) {\small$\{ w_0 \}$};

         \draw (id) to (s);
         \draw (s) to (C);
         \draw (C) to (w0);
      \end{tikzpicture} \\
   \end{longtable}
\end{center}

\subsubsection{Kazhdan-Lusztig cells do not decompose into \texorpdfstring{$p$}{p}-cells}
\label{secDecompCounterex}

In this section, we will present the smallest example where Kazhdan-Lusztig cells
do not decompose into $p$-cells. This happens in type $C_3$ for $p=2$. We 
label the simple reflections as follows (also note their colours):
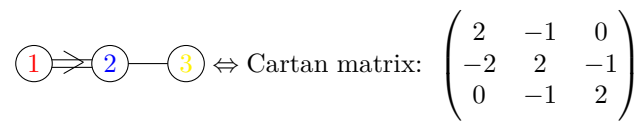
\[\begin{tikzpicture}[auto, baseline=(1.base)]
   \draw (1,0) -- (0,0);
   \draw (0,\edgeShift) -- (-1,\edgeShift);
   \draw (0,-\edgeShift) -- (-1,-\edgeShift);
   \path (-1,0) to node[Greater] (mid) {} (0,0);
   \draw (mid.center) to +(150:\wingLen);
   \draw (mid.center) to +(210:\wingLen);
   \node [DynNode] (1) at (-1,0) {$\color{red}{1}$};
   \node [DynNode] (2) at (0,0) {$\color{blue}{2}$};
   \node [DynNode] (3) at (1,0) {$\color{yellow}{3}$};
\end{tikzpicture} 
\Leftrightarrow \text{Cartan matrix: }
\begin{pmatrix}
   2 & -1 & 0 \\
   -2 & 2 & -1 \\
   0 & -1 & 2
\end{pmatrix}
\]

Explicit computer calculation gives the following Kazhdan-Lusztig right cells:

\begin{align*}
   C_0 &= \{ \id \} \\
   C_1 &= \{ 1, 12, 121, 123 \} \\
   C_2 &= \{ 2, 21, 23, 212, 2123 \} \\
   C_3 &= \{ 3, 32, 321, 3212, 32123 \} \\
   C_4 &= \{ 13, 132, 1321 \}\\
   C_5 &= \{ 213, 2132, 21321 \} \\
   C_6 &= \{ 232, 2321, 23212 \} \\
   C_{7} &= \{ 2121, 21213, 212132, 2121321, 21213213 \} \\   
   C_8 &= \{ 1213, 12132, 121321 \} \\
   C_9 &= \{ 1232, 12321, 123212 \} \\
   C_{10} &= \{ 13212, 132123, 1213212, 1232123, 12132123 \} \\
   C_{11} &= \{ 21232, 212321, 2123212 \} \\
   C_{12} &= \{ 232123, 232121, 2321213, 23212132 \} \\
   C_{13} &= \{ w_0 \}
\end{align*}

For $p=2$ these right Kazhdan-Lusztig cells exhibit the following decomposition
behaviour into right $p$-cells:
\begin{align*}
   C_2 &= \underbrace{\{ 2, 21\}}_{pC_{2'}} \cup \underbrace{\{23, 212, 2123 \}}_{pC_{2''}} \\
   C_3 &= \underbrace{\{ 3, 32\}}_{pC_{3'}} \cup \underbrace{\{ 321, 3212, 32123 \}}_{pC_{3''}} \\
   C_6 \cup C_{12} &= \underbrace{\{ 232 \}}_{pC_6} \cup \underbrace{\{ 2321, 23212, 232123 \}}_{pC_{6/12}}
                     \cup \underbrace{\{ 232121, 2321213, 23212132 \}}_{pC_{12}} \\
   C_i &= pC_i \text{ for } i \in \{0, \dots, 13\} \setminus \{2,3,6,12\}
\end{align*}

The Hasse-diagrams of the cell preorders look as follows. We display Kazhdan-Lusztig
right cells on the left and right $p$-cells on the right. In these diagrams the cells 
that are depicted at one height form a two-sided cell.
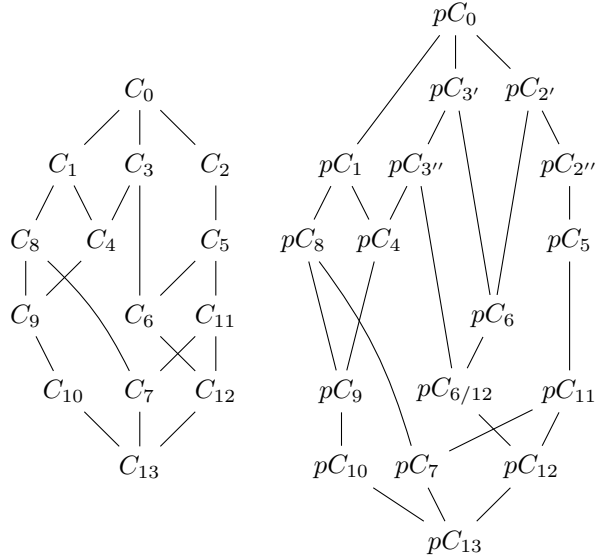
\[
\begin{tikzpicture}[auto, baseline=(current  bounding  box.center)]
   \node (0) at (0, 2.5) {$C_0$};
   \node (1) at (-1, 1.5) {$C_1$};
   \node (2) at (1, 1.5) {$C_2$};
   \node (3) at (0, 1.5) {$C_3$};
   \node (4) at (-0.5, 0.5) {$C_4$};
   \node (5) at (1, 0.5) {$C_5$};
   \node (8) at (-1.5, 0.5) {$C_8$};
   \node (9) at (-1.5, -0.5) {$C_9$};
   \node (11) at (1, -0.5) {$C_{11}$};
   \node (6) at (0, -0.5) {$C_6$};
   \node (10) at (-1, -1.5) {$C_{10}$};
   \node (7) at (0, -1.5) {$C_7$};
   \node (12) at (1, -1.5) {$C_{12}$};
   \node (13) at (0, -2.5) {$C_{13}$};
   \draw (0) to (1);
   \draw (0) to (2);
   \draw (0) to (3);
   \draw (1) to (8);
   \draw (1) to (4);
   \draw (3) to (4);
   \draw (3) to (6);
   \draw (2) to (5);
   \draw (4) to (9);
   \draw (5) to (6);
   \draw (5) to (11);
   \draw (8) to (9);
   \draw[bend left=10] (8) to (7);
   \draw (6) to (12);
   \draw (9) to (10);
   \draw (11) to (12);
   \draw (11) to (7);
   \draw (10) to (13);
   \draw (7) to (13);
   \draw (12) to (13);
\end{tikzpicture}
\quad
\begin{tikzpicture}[auto, baseline=(current  bounding  box.center)]
   \node (0) at (0, 3.5) {$pC_0$};
   \node (3') at (0, 2.5) {$pC_{3'}$};
   \node (2') at (1, 2.5) {$pC_{2'}$};
   \node (1) at (-1.5, 1.5) {$pC_1$};
   \node (3'') at (-0.5, 1.5) {$pC_{3''}$};
   \node (2'') at (1.5, 1.5) {$pC_{2''}$};
   \node (8) at (-2.0, 0.5) {$pC_8$};
   \node (4) at (-1.0, 0.5) {$pC_4$};
   \node (5) at (1.5, 0.5) {$pC_5$};   
   \node (6) at (0.5, -0.5) {$pC_6$};
   \node (9) at (-1.5, -1.5) {$pC_9$};
   \node (6/12) at (0, -1.5) {$pC_{6/12}$};
   \node (11) at (1.5, -1.5) {$pC_{11}$};
   \node (10) at (-1.5, -2.5) {$pC_{10}$};
   \node (7) at (-0.5, -2.5) {$pC_7$};
   \node (12) at (1, -2.5) {$pC_{12}$};
   \node (13) at (0, -3.5) {$pC_{13}$};
   \draw (0) to (1);
   \draw (0) to (2');
   \draw (0) to (3');
   \draw (3') to (3'');
   \draw (3') to (6);
   \draw (2') to (2'');
   \draw (2') to (6);
   \draw (1) to (8);
   \draw (1) to (4);
   \draw (3'') to (4);
   \draw (3'') to (6/12);
   \draw (2'') to (5);
   \draw (4) to (9);
   \draw (5) to (11);
   \draw (8) to (9);
   \draw[bend left=10] (8) to (7);
   \draw (6) to (6/12);
   \draw (6/12) to (12);
   \draw (9) to (10);
   \draw (11) to (12);
   \draw (11) to (7);
   \draw (10) to (13);
   \draw (7) to (13);
   \draw (12) to (13);
\end{tikzpicture}\]

Finally, let us try to explain the non-trivial decomposition behaviour. For the elements
in $C_6 \cup C_{12}$ we have:
\begin{align*}
   \pcan[2]{23212} &= \kl{23212} + \kl{232} \\
   \pcan[2]{232123} &= \kl{232123} + (v+ v^{-1}) \kl{232} \\
   \pcan[2]{23212132} &= \kl{23212132} + \kl{232123} \\
   \pcan[2]{x} &= \kl{x} \text{ for } x \in (C_6 \cup C_{12}) \setminus \{23212, 232123, 23212132\}
\end{align*}

The subquotient $\heck[\cle{R} C_{6}] / \heck[\clt{R} C_{12}]$ is a module for the
Hecke algebra and the action on the $2$-canonical basis of this module can be encoded
in the following graph:
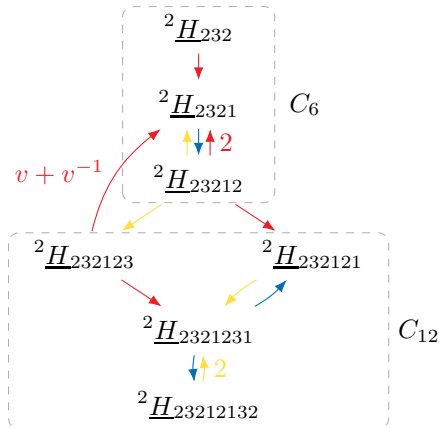
\[ \begin{tikzpicture}[auto, baseline=(current  bounding  box.center)]
   \node (232) at (0, 2.5) {$\pcan[2]{232}$};
   \node (2321) at (0, 1.5) {$\pcan[2]{2321}$};
   \node (23212) at (0, 0.5) {$\pcan[2]{23212}$};
   \node (232123) at (-1.5, -0.5) {$\pcan[2]{232123}$};
   \node (2321231) at (0, -1.5) {$\pcan[2]{2321231}$};
   \node (23212312) at (0, -2.5) {$\pcan[2]{23212132}$};
   \node (232121) at (1.5, -0.5) {$\pcan[2]{232121}$};
   \draw[-Latex, s] (232) to (2321);
   \draw[-Latex, t] (2321) to (23212);
   \draw[-Latex, transform canvas={xshift=-1ex}, u] (23212) to (2321);
   \draw[-Latex, transform canvas={xshift=1ex}, s] (23212) to node[labelling, midway, label=right:$2$]{} (2321);
   \draw[-Latex, u] (23212) to (232123);
   \draw[-Latex, s] (23212) to (232121);
   \draw[-Latex, s] (232123) to (2321231);
   \draw[-Latex, bend left=20, s] (232123) to node[labelling, midway, label=left:$v+v^{-1}$]{} (2321);
   \draw[-Latex, bend right=10, t] (2321231) to (23212312);
   \draw[-Latex, bend right=10, u] (23212312) to node[labelling, midway, label=right:$2$]{} (2321231);
   \draw[-Latex, bend right=10, u] (232121) to (2321231);
   \draw[-Latex, bend right=10, t] (2321231) to (232121);
   \draw[rounded corners, dashed, color=black!30] (-1, 0.2) rectangle (1, 2.8) {};
   \draw[rounded corners, dashed, color=black!30] (-2.5, -0.2) rectangle (2.5, -2.8) {};
   \node (C6) at (1.4, 1.5) {$C_6$};
   \node (C12) at (2.9, -1.5) {$C_{12}$};
\end{tikzpicture}\]
Note that we omitted all edge labels equal to $1$ and all loops labelled with $v+v^{-1}$.
The strongly connected components of this graph give the right $p$-cells $pC_6$,
$pC_{6/12}$ and $pC_{12}$. From this we see that neither two-sided nor right Kazhdan-Lusztig
cells decompose into the corresponding $p$-cells in this example.

In this case we cannot apply the decomposition criterion from \cref{secDecompCriterion}
because the Kazhdan-Lusztig right cell $C_{12}$ does not satisfy the assumptions
of \cref{corCoeff}.

In type $B_3$ the right (and two-sided) Kazhdan-Lusztig cells do decompose
into right (resp. two-sided) $p$-cells, whereas in type $B_4$ they do not. 
This calculation together with \cref{propNoDecomp} below and \cref{corLCellsDecomp}
completely settles the  question of when right Kazhdan-Lusztig cells 
decompose into right $p$-cells in types $B$ and $C$. In summary, right 
Kazhdan-Lusztig cells in types $B_n$ and $C_n$ decompose into right $p$-cells 
for $p > 2$ or $p=2$ and $n \leqslant 3$ in type $B_n$.

\begin{prop}
   \label{propNoDecomp}
   Suppose the Kazhdan-Lusztig right cells of $H$ do not decompose
   into right $p$-cells for some $p > 0$. Let $W$ be a 
   crystallographic Coxeter group that contains $H$ as standard
   parabolic subgroup. Then the Kazhdan-Lusztig right cells of $W$
   do not decompose into $p$-cells.
\end{prop}
\begin{proof}
   By assumption we may find two distinct right Kazhdan-Lusztig cells 
   $C_1, C_2 \subseteq H$ and a right $p$-cell $\p{C} \subseteq H$ for 
   some $p$ such that $\p{C} \cap C_1$ and $\p{C} \cap C_2$ are non-empty. 
   By \cref{thmParaCombPCells} we may find right Kazhdan-Lusztig cells 
   $\widetilde{C_1}, \widetilde{C_2} \subseteq W$ and a right $p$-cell
   $\widetilde{\p{C}} \subseteq W$ such that 
   $C_i \subseteq \widetilde{C_i}$ for $i \in \{1, 2\}$ and 
   $\p{C} \subseteq \widetilde{\p{C}}$.
   
   Observe that for $i \in \{1,2\}$ the set $\widetilde{\p{C}} \cap \widetilde{C_i}$ 
   is non-empty as it contains $\p{C} \cap C_i$. Therefore, it is enough to show 
   that $\widetilde{C_1}$ and $\widetilde{C_2}$ are disjoint. Let $X$ be the 
   set of representatives of minimal length of cosets in  $H\backslash W$. 
   To conclude we need to recall a result about the induction of Kazhdan-Lusztig 
   cells (see \cite[Theorem 1]{GeIndKLCells} for the version for left cells):
   
   \begin{thm}[Induction of right Kazhdan-Lusztig cells]
      Let $H \subseteq W$ be a standard parabolic subgroup, and let $C$ be a
      right Kazhdan-Lusztig cell of $H$. Then $C \cdot X$ is a union of right cells 
      of $W$.
   \end{thm}
   
   Since we have $\widetilde{C_i} \subseteq C_i \cdot X$ for $i \in \{1,2\}$, 
   the claim follows from the fact that $C_1 \cdot X$ and $C_2 \cdot X$ are disjoint.
\end{proof}

%
%

\section{Left and Right Star Operations}
\label{secStarOps}

\subsection{Definition and Numerical Consequences}

In the section, we will prove consequences of the Kazhdan-Lusztig star-operations
for the $p$-canonical basis. The star-operations were originally introduced in 
\cite[\S4]{KL}, generalizing (dual) Knuth operations from the symmetric group to
pairs of simple reflections $r, t \in S$ in general Coxeter groups with
$m_{r,t} = 3$. In the literature, there does not seem to exist a consensus on how
to generalize the star-operations to the case $3 < m_{r,t} < \infty$. 
We propose the following generalization as in \cite[Remark 4.3]{BGTauInv}: 

\begin{defn}
   \label{defStringsStars}
   Let $r, t \in S$ be two simple reflections. Define:
   \begin{align*}
      \mathcal{D}_L(r, t) &\defeq \{ w \in W \; \vert \;  \lvert \desc{L}(w) \cap \{r, t\} \rvert = 1 \} \\
      \mathcal{D}_R(r, t) &\defeq \{ w \in W \; \vert \;  \lvert \desc{R}(w) \cap \{r, t\} \rvert = 1 \}
   \end{align*}
   Set $m \defeq m_{r, t}$.  For $1 \leqslant k \leqslant m$ denote by ${}_r\hat{k}= rtrt\dots$ the alternating word 
   in $r$ and $t$ starting in $r$ of length $k$. Recall that $W^{\{r, t\}}$
   denotes the set of representatives of minimal length of cosets in 
   $W / \langle r, t\rangle$ (see \cref{defFinSubset}). Any coset in 
   $W/\langle r, t \rangle$ contains a unique element 
   $\widetilde{w} \in W^{\{r, t\}}$ and can be partitioned in 
   the following sets:
   \[ \begin{cases}
         \{ \widetilde{w} \} \cup \{ \widetilde{w}\cdot{}_r\hat{k} \; \vert \; 1 \leqslant k < m \} \cup
         \{ \widetilde{w}\cdot{}_t\hat{k} \; \vert \; 1 \leqslant k < m \} \cup \{ \widetilde{w}\cdot {}_t \hat{m} \}
            & \text{if } m < \infty \text{,} \\
         \{ \widetilde{w} \} \cup \{ \widetilde{w}\cdot{}_r\hat{k} \; \vert \; 1 \leqslant k \} \cup
         \{ \widetilde{w}\cdot{}_t\hat{k} \; \vert \; 1 \leqslant k \}
            & \text{if } m = \infty\text{.} 
      \end{cases}
   \]
   For $m < \infty$ the element $\widetilde{w}\cdot {}_t \hat{m}$ is the unique 
   element of maximal length in the coset. The set 
   $\{ \widetilde{w}\cdot {}_x\hat{k} \; \vert \; 1 \leqslant k < m \}$
   for some $x \in \{r, t\}$ is called a \emph{right $\langle r, t \rangle$-string}
   (also for $m = \infty$) and contained in $\mathcal{D}_R(r, t)$. The element 
   $\widetilde{w}\cdot {}_x\hat{k}$ is the \emph{$k$-th element} in this string.
   
   It is easy to see that actually any element $w \in \mathcal{D}_R(r,t)$  lies in
   a right $\langle r, t \rangle$-string and can thus be written as $\widetilde{w}\cdot {}_x\hat{k}$ 
   for some $x \in \{r, t\}$ and $1 \leqslant k < m$ where $\widetilde{w}$ is 
   the element of minimal length in the right coset 
   $w\langle r, t\rangle \in W/\langle r,t \rangle$.
   
   Assume $3 \leqslant m < \infty$. Then the right star operation 
   $(-)^{\ast}$ is an involution on $\mathcal{D}_R(r,t)$ sending 
   $w = \widetilde{w}\cdot {}_x\hat{k}$ for $x$ and $k$ as above to $\widetilde{w}\cdot
   {}_x\widehat{(m-k)}$. The left star operation $\pre{\ast}{(-)}$ is an 
   involution on $\mathcal{D}_L(r,t)$ defined analogously. 
\end{defn}
   
It follows immediately that the left and right star operations are related via:
\[\pre{\ast}{w} = (( w^{-1})^{\ast})^{-1} \text{.}\]

Fix for the rest of the section two simple reflections $r, t \in S$ with 
$3 \leqslant m \defeq m_{r, t} < \infty$. The multiplication formula for the
Kazhdan-Lusztig basis does not easily generalize to the 
$p$-canonical basis. However it will still be important to understand 
the structure coefficients of the $p$-canonical basis a little bit better. The 
next lemma states a crucial observation that will be used frequently below.
\begin{lem}
   \label{lemKLinPCanStringCoeff}
   Let $x, z \in \mathcal{D}_R(r,t)$ with $r \in \desc{R}(x)$. The coefficient of 
   $\kl{z}$ in $\pcan{x} \kl{t}$ is given by
   \[ \delta_{zr \in \mathcal{D}_R(r, t)} \pre{p}{m}_{zr, x} + \delta_{zt \in 
      \mathcal{D}_R(r, t)} \pre{p}{m}_{zt, x}
   \]
   where $\delta_{zr \in \mathcal{D}_R(r, t)}$ is the Kronecker delta.
\end{lem}
\begin{proof}
   Rewrite the product
   $\pcan{x} \kl{t}$ as follows:
      \begin{align*}
      \pcan{x} \kl{t} &= (\kl{x} + \sum_{w < x} \pre{p}{m}_{w, x} \kl{w}) \kl{t} \\
      &= \kl{xt} + \sum_{\substack{z \leqslant x \\ zt < z}} \mu(z, x) \kl{z}
      + \sum_{\substack{w < x\\wt > w}} \pre{p}{m}_{w, x} 
      \left(\kl{wt} + \sum_{\substack{z \leqslant w\\ zt < z}} \mu(z,w) \kl{z}\right)\\
      &\phantom{=} + \sum_{\substack{z < x\\ zt < z}}(v+v^{-1})\pre{p}{m}_{z, x}\kl{z}
   \end{align*}
   If $t$ is not in the right descent set of $z$, the coefficient in front of
   $\kl{z}$ has to be zero. 
   By \cref{propPCanProps}(v) the formula stated
   above also gives zero in this case. Thus, we assume $t \in \desc{R}(z)$ from now on.
   Since $z$ lies in $\mathcal{D}_R(r,t)$, this implies that $r$ does not lie in the
   right descent set of $z$. Consider an element $w \in W$ with $\pre{p}{m}_{w, x} \ne 0$ 
   such that $\kl{z}$ occurs with non-zero coefficient in $\kl{w} \kl{t}$. Observe that 
   $w$ could be $x$. By \cref{propPCanProps}(v) $\pre{p}{m}_{w, x} \ne 0$ implies 
   $r \in \desc{R}(x) \subseteq \desc{R}(w)$. Since the right descent sets differ,
   $w$ and $z$ cannot coincide. In particular, $w$ does not have $t$ in its right descent
   set and thus also lies in $\mathcal{D}_R(r,t)$. (Otherwise $\kl{z}$ could not
   occur with non-zero coefficient in $\kl{w}\kl{t}$.) Recall the following important fact 
   about the $\mu$-coefficients (from \cite[(2.3.f)]{KL}):
   \begin{lem}
      Let $z < w \in W$ and $r \in \desc{R}(w) \setminus \desc{R}(z)$. Then we have:
      \[ \mu(z, w) \ne 0 \Leftrightarrow w = zr \]
      Moreover, $\mu(z, w) = 1$ in this case.
   \end{lem}
   If $z < w$ holds, then we may apply this lemma to $z < w$ and the simple reflection 
   $r$ to get that $w = zr$. Otherwise, we have $z = wt > w$ (due to the 
   multiplication formula from \cite[(2.3.b)]{KL}) and again $\mu(w, z) = 1$.
   In both cases, we see that $z$ and $w$ lie in the same right $\langle r, t \rangle$-string 
   and the coefficient of $\kl{z}$ in $\pcan{x} \kl{t}$ is 
   \[\delta_{zr \in \mathcal{D}_R(r, t)} \pre{p}{m}_{zr, x} + \delta_{zt \in 
      \mathcal{D}_R(r, t)} \pre{p}{m}_{zt, x} \text{.}\qedhere\]
\end{proof}

\begin{defn}
   The \emph{weak right Bruhat graph} of $(W, S)$ is the labelled, directed graph
   with vertex set $W$ and edge set
   \[ \{ (w, ws) \; \vert \; w \in W, s \in S \setminus \desc{R}(w) \} \text{.} \]
   For $w \in W$ and $s \in S \setminus \desc{R}(w)$ the edge
   $(w, ws)$ is labelled by $\rt[s]$.
\end{defn}

The reader may picture the formula from \cref{lemKLinPCanStringCoeff} as follows: 
Consider the subgraph of the weak right Bruhat graph on 
the vertices $\mathcal{D}_R(r,t) \cap (z \langle r, t \rangle)$ and only edges labelled
by $\rt[r]$ or $\rt[t]$. In order to get the coefficient of $\kl{z}$ in $\pcan{x} \kl{t}$
we simply have to slide the coefficients $\pre{p}{m}_{?, x}$ up along an edge
labelled by $\rt[t]$ and down along an edge labelled by $\rt[r]$ and sum them up
if two coefficients collide at a vertex in the process. Here, up and down
are meant with respect to the weak right Bruhat order.

For the rest of the section, we will assume:
\[ p > \begin{cases}
          1 & \text{if } m = 3\text{,} \\
          2 & \text{if } m = 4\text{,} \\
          3 & \text{if } m = 6\text{.}
       \end{cases}\]
This ensures that for $w \in \langle r, t\rangle$ we have $\pcan{w} = \kl{w}$ 
(see \cite[\S5.1, \S5.2 and \S5.6]{JW}) and thus Kazhdan-Lusztig cells and $p$-cells
in $\langle r, t\rangle$ coincide.

\begin{prop}
   \label{propBCCoeffRels}
   Let $\sigma_x = \{ x_1 < x_2 < \dots \}$ and $\sigma_z = \{ z_1 < z_2 < \dots \}$
   be two right $\langle r, t \rangle$-strings consisting of $m - 1$ elements. 
   Then we have the following relations among the coefficients $\pre{p}{m}_{z_j, x_i}$:
   \begin{align}
      m = 3 &\Rightarrow \begin{cases}
                           \pre{p}{m}_{z_1, x_1} = \pre{p}{m}_{z_2, x_2} \\
                           \pre{p}{m}_{z_2, x_1} = \pre{p}{m}_{z_1, x_2}
                        \end{cases} \\
      m = 4 &\Rightarrow \begin{cases}
                           \pre{p}{m}_{z_1, x_1} = \pre{p}{m}_{z_3, x_3} \\
                           \pre{p}{m}_{z_2, x_1} = \pre{p}{m}_{z_1, x_2} = \pre{p}{m}_{z_3, x_2} = \pre{p}{m}_{z_2, x_3} \\
                           \pre{p}{m}_{z_3, x_1} = \pre{p}{m}_{z_1, x_3} \\
                           \pre{p}{m}_{z_2, x_2} = \pre{p}{m}_{z_1, x_1} + \pre{p}{m}_{z_3, x_1}
                         \end{cases} \\
      m = 6 &\Rightarrow \begin{cases}
                           \pre{p}{m}_{z_1, x_1} = \pre{p}{m}_{z_5, x_5} \\
                           \pre{p}{m}_{z_2, x_1} = \pre{p}{m}_{z_1, x_2} = \pre{p}{m}_{z_5, x_4} = \pre{p}{m}_{z_4, x_5} \\
                           \pre{p}{m}_{z_3, x_1} = \pre{p}{m}_{z_1, x_3} = \pre{p}{m}_{z_5, x_3} = \pre{p}{m}_{z_3, x_5} \\
                           \pre{p}{m}_{z_4, x_1} = \pre{p}{m}_{z_1, x_4} = \pre{p}{m}_{z_5, x_2} = \pre{p}{m}_{z_2, x_5} \\
                           \pre{p}{m}_{z_5, x_1} = \pre{p}{m}_{z_1, x_5} \\
                           \pre{p}{m}_{z_2, x_2} = \pre{p}{m}_{z_4, x_4} = \pre{p}{m}_{z_1, x_1} + \pre{p}{m}_{z_3, x_1} \\
                           \pre{p}{m}_{z_3, x_2} = \pre{p}{m}_{z_2, x_3} = \pre{p}{m}_{z_4, x_3} = \pre{p}{m}_{z_3, x_4} = \pre{p}{m}_{z_2, x_1} + \pre{p}{m}_{z_4, x_1} \\
                           \pre{p}{m}_{z_4, x_2} = \pre{p}{m}_{z_2, x_4} = \pre{p}{m}_{z_3, x_1} + \pre{p}{m}_{z_5, x_1} \\
                           \pre{p}{m}_{z_3, x_3} = \pre{p}{m}_{z_1, x_1} + \pre{p}{m}_{z_3, x_1} + \pre{p}{m}_{z_5, x_1}
                         \end{cases}
   \end{align}
\end{prop}
\begin{proof}
   Comparing Laurent polynomials coefficient-wise induces a partial order which we will
   use in the following. For $x \in \sigma_x$ with $r \in \desc{R}(x)$, rewrite 
   $\pcan{x} \kl{t}$ in terms of the $p$-canonical basis to get:
   \[ \sum_{v \leqslant xt} \pre{p}{\mu}^v_{x, t} \pcan{v} \]
   Express this in the Kazhdan-Lusztig basis and use \cref{lemKLinPCanStringCoeff} 
   to see that we have the following inequality for $z \in \sigma_z$ which
   is actually satisfied with equality:
   \begin{equation}
   \label{eqnCoeffIneq}
      \sum_{z \leqslant v \leqslant xt} \pre{p}{\mu}^v_{x, t} \pre{p}{m}_{z, v} 
      \leqslant \delta_{zr \in \mathcal{D}_R(r, t)}  \pre{p}{m}_{zr, x} + 
      \delta_{zt \in \mathcal{D}_R(r, t)} \pre{p}{m}_{zt, x}    
   \end{equation}
  
   We want to use a weaker form of this inequality together with the fact that we understand the structure 
   coefficients in the right $\langle r, t \rangle$-coset of $x$ (see \cref{corParStrCoeffs}). 
   Let $s_i \in \{r, t\}$ be the simple reflection such that $x_i s_i > x_i$ for 
   $1 \leqslant i \leqslant m - 1$. Write $x_0$ (resp. $x_m$) for the shortest 
   (resp. longest) element in the right $\langle r, t \rangle$-coset of $x$. Similarly for
   $z_0$ and $z_m$. We can restrict the sum on the left hand side 
   of \eqref{eqnCoeffIneq} to $v \in \sigma_x$  
   and $v \leqslant xt$. From \cref{corParStrCoeffs} and the explicit knowledge of 
   the structure constants of the Kazhdan-Lusztig basis in the dihedral case we deduce:
   \[\pcan{x_i} \pcan{s_i} = 
   \begin{cases} 
      \pcan{x_2} & \text{if } i = 1\text{,}\\
      \pcan{x_{i+1}} + \pcan{x_{i-1}} & \text{otherwise.}\\
   \end{cases} \quad \left(\text{mod } \sum_{\substack{w < x_{i+1}\\ 
                     w \notin \sigma_x}} \Z[v, v^{-1}] \pcan{w} \right)\]
   Using this in inequality \eqref{eqnCoeffIneq} and letting $x$ and $z$ in their right 
   $\langle r, t \rangle$-string vary, we obtain for $1 \leqslant i, j \leqslant m - 1$:
   \begin{equation}
      \resizebox{0.9\linewidth}{!}{
      $\delta_{x_{i+1} \in \sigma_x} \pre{p}{m}_{z_j, x_{i+1}} + \delta_{x_{i-1} \in \sigma_x}
      \pre{p}{m}_{z_j, x_{i-1}} \leqslant \delta_{z_{j+1} \in \sigma_z}  \pre{p}{m}_{z_{j+1}, x_i} + 
      \delta_{z_{j-1} \in \sigma_z} \pre{p}{m}_{z_{j-1}, x_i}$} \tag{$\ast_{i,j}$}
   \end{equation}
  
   Surprisingly enough, any solution to this system of inequalities satisfies
   all inequalities with equality. We will solve this system of linear inequalities
   for $m = 6$ and leave the cases $m \in \{3, 4\}$ to the reader. To simplify notation, 
   write $a_{j,i} = \pre{p}{m}_{z_j, x_i}$ for all $1 \leqslant i, j \leqslant m-1$
   and view them as indeterminates. The set of inequalities can be partitioned 
   into two sets of inequalities which can be solved completely independently:
   \[ \{ (\ast_{i,j}) \; \vert \; i + j \text{ even} \} \cup \{ (\ast_{i,j}) \; \vert \; i + j \text{ odd} \} \]
   
   First, let us consider $\{ (\ast_{i,j}) \; \vert \; i + j \text{ even} \}$:
   \begin{align}
      a_{1,2} &\leqslant a_{2,1} \tag{i} \label{eqnEven1} \\
      a_{3,2} &\leqslant a_{2,1} + a_{4,1} \tag{ii} \label{eqnEven2} \\
      a_{5,2} &\leqslant a_{4,1} \tag{iii} \label{eqnEven3} \\
      a_{2,3} + a_{2,1} &\leqslant a_{1,2} + a_{3,2} \tag{iv} \label{eqnEven4} \\
      a_{4,3} + a_{4,1} &\leqslant a_{3,2} + a_{5,2} \tag{v} \label{eqnEven5} \\
      a_{1,4} + a_{1,2} &\leqslant a_{2,3} \tag{vi} \label{eqnEven6} \\
      a_{3,4} + a_{3,2} &\leqslant a_{2,3} + a_{4,3} \tag{vii} \label{eqnEven7} \\
      a_{5,4} + a_{5,2} &\leqslant a_{4,3} \tag{viii} \label{eqnEven8} \\
      a_{2,5} + a_{2,3} &\leqslant a_{1,4} + a_{3,4} \tag{ix} \label{eqnEven9} \\
      a_{4,5} + a_{4,3} &\leqslant a_{3,4} + a_{5,4} \tag{x} \label{eqnEven10} \\
      a_{1,4} &\leqslant a_{2,5} \tag{xi} \label{eqnEven11} \\
      a_{3,4} &\leqslant a_{2,5} + a_{4,5} \tag{xii} \label{eqnEven12} \\
      a_{5,4} &\leqslant a_{4,5} \tag{xiii} \label{eqnEven13}
   \end{align}
   Consider the following chain of inequalities:
   \[ a_{2,1} \overset{\eqref{eqnEven4}}{\leqslant} a_{1,2} + a_{3,2} - a_{2,3} 
      \overset{\eqref{eqnEven7}}{\leqslant} a_{1,2} + a_{4,3} - a_{3,4} 
      \overset{\eqref{eqnEven10}}{\leqslant} a_{1,2} + a_{5,4} -a_{4,5}
      \overset{\eqref{eqnEven13}}{\leqslant} a_{1,2}
      \overset{\eqref{eqnEven1}}{\leqslant} a_{2,1}\]
   This implies that the inequalities \eqref{eqnEven1}, \eqref{eqnEven4}, \eqref{eqnEven7}, 
   \eqref{eqnEven10} and \eqref{eqnEven13} are all satisfied with equality, which in turn implies:
   \[ a_{2,3} = a_{3,2} \qquad a_{3,4} = a_{4,3} \]
   Next, consider the following chain:
   \[ a_{4,1} \overset{\eqref{eqnEven5}}{\leqslant} a_{3,2} + a_{5,2} - a_{4,3}
      \overset{\eqref{eqnEven7}}{=} a_{2,3} - a_{3,4} + a_{5,2}
      \overset{\eqref{eqnEven9}}{\leqslant} a_{1,4} - a_{2,5} + a_{5,2}
      \overset{\eqref{eqnEven11}}{\leqslant} a_{5,2}
      \overset{\eqref{eqnEven3}}{\leqslant} a_{4,1} \]
   This shows that the inequalities \eqref{eqnEven3}, \eqref{eqnEven5}, \eqref{eqnEven9} and 
   \eqref{eqnEven11} are also satisfied with equality, from which we deduce:
   \[ a_{2,3} = a_{3,2} = a_{3,4} = a_{4,3} \]
   Finally, we have
   \[ a_{1,4} + a_{1,2} \overset{\eqref{eqnEven6}}{\leqslant} a_{2,3} = a_{3,2} 
   \overset{\eqref{eqnEven2}}{\leqslant} a_{2,1} + a_{4,1} \]
   and 
   \[ a_{5,4} + a_{5,2} \overset{\eqref{eqnEven8}}{\leqslant} a_{4,3} = a_{3,4} 
   \overset{\eqref{eqnEven12}}{\leqslant} a_{2,5} + a_{4,5} \]
   which imply using $a_{1,2} = a_{2,1}$ and $a_{5,4} = a_{4,5}$ respectively
   \[ a_{1,4} \leqslant a_{4,1} \qquad a_{5,2} \leqslant a_{2,5} \]
   Using $a_{1,4} = a_{2,5}$ and $a_{4,1} = a_{5,2}$ finishes the argument.

   Next, we solve $\{ (\ast_{i,j}) \; \vert \; i + j \text{ odd} \}$:
   \begin{align}
      a_{2,2} &\leqslant a_{1,1} + a_{3,1} \tag{i'} \label{eqnOdd1} \\
      a_{4,2} &\leqslant a_{3,1} + a_{5,1} \tag{ii'} \label{eqnOdd2} \\
      a_{1,3} + a_{1,1} &\leqslant a_{2,2} \tag{iii'} \label{eqnOdd3} \\
      a_{3,3} + a_{3,1} &\leqslant a_{2,2} + a_{4,2} \tag{iv'} \label{eqnOdd4} \\
      a_{5,3} + a_{5,1} &\leqslant a_{4,2} \tag{v'} \label{eqnOdd5} \\
      a_{2,4} + a_{2,2} &\leqslant a_{1,3} + a_{3,3}  \tag{vi'} \label{eqnOdd6} \\
      a_{4,4} + a_{4,2} &\leqslant a_{3,3} + a_{5,3} \tag{vii'} \label{eqnOdd7} \\
      a_{1,5} + a_{1,3} &\leqslant a_{2,4} \tag{viii'} \label{eqnOdd8} \\
      a_{3,5} + a_{3,3} &\leqslant a_{2,4} + a_{4,4} \tag{ix'} \label{eqnOdd9} \\
      a_{5,5} + a_{5,3} &\leqslant a_{4,4} \tag{x'} \label{eqnOdd10} \\
      a_{2,4} &\leqslant a_{1,5} + a_{3,5} \tag{xi'} \label{eqnOdd11} \\
      a_{4,4} &\leqslant a_{3,5} + a_{5,5} \tag{xii'} \label{eqnOdd12}
   \end{align}
   In this case we argue as follows:
   \[ a_{1,1} \overset{\eqref{eqnOdd3}}{\leqslant} a_{2,2} - a_{1,3}  
      \overset{\eqref{eqnOdd6}}{\leqslant} a_{3,3} - a_{2,4} 
      \overset{\eqref{eqnOdd9}}{\leqslant} a_{4,4} - a_{3,5}
      \overset{\eqref{eqnOdd12}}{\leqslant} a_{5,5} \]
   We use this in the last inequality of the following chain:\\
   \resizebox{\linewidth}{!}{
      $a_{3,1} \overset{\eqref{eqnOdd4}}{\leqslant} a_{2,2} + a_{4,2} - a_{3,3} 
      \overset{\eqref{eqnOdd7}}{\leqslant} a_{2,2} + a_{5,3} - a_{4,4} 
      \overset{\eqref{eqnOdd10}}{\leqslant} a_{2,2} - a_{5,5}
      \overset{\eqref{eqnOdd1}}{\leqslant} a_{1,1} + a_{3,1} - a_{5,5}
      \leqslant a_{3,1}$}
   This implies that $a_{1,1} = a_{5,5}$ and the inequalities \eqref{eqnOdd1}, \eqref{eqnOdd3},
   \eqref{eqnOdd4}, \eqref{eqnOdd6}, \eqref{eqnOdd7}, \eqref{eqnOdd9}, \eqref{eqnOdd10}
   and \eqref{eqnOdd12} are satisfied with equality. Moreover, we have the following
   equivalences:
   \begin{align*}
      \eqref{eqnOdd1} = \eqref{eqnOdd3} &\Leftrightarrow a_{1,3} = a_{3,1}  \Leftrightarrow 
      \eqref{eqnOdd4} = \eqref{eqnOdd6} \\
      &\Leftrightarrow a_{2,4} = a_{4,2} \Leftrightarrow \eqref{eqnOdd7} = \eqref{eqnOdd9} \\
      &\Leftrightarrow a_{5,3} = a_{3,5} \Leftrightarrow \eqref{eqnOdd10} = \eqref{eqnOdd12}
   \end{align*}
   The last four inequalities can be used as follows:
   \begin{align*}
      a_{5,3} + a_{5,1} \overset{\eqref{eqnOdd5}}{\leqslant} &a_{4,2} 
      \overset{\eqref{eqnOdd2}}{\leqslant} a_{3,1} + a_{5,1} \\
      a_{1,5} + a_{1,3} \overset{\eqref{eqnOdd8}}{\leqslant} &a_{2,4}
      \overset{\eqref{eqnOdd11}}{\leqslant} a_{1,5} + a_{3,5} 
   \end{align*}
   This gives $a_{5,3} \leqslant a_{3,1} = a_{1,3} \leqslant a_{3,5} = a_{5,3}$ and
   finishes the argument.
   
   Finally, observe that the space of solutions for these (in)equalities is a free 
   $\Z[v,v^{-1}]$-module of rank $m-1$ and we can choose $\{a_{1,1}, a_{2,1}, \dots,
   a_{m-1,1}\}$ as a basis. In other words, the solution is uniquely determined after
   fixing these Laurent polynomials. From all these equalities, the reader can easily 
   deduce the relations given in the proposition where we expressed each coefficient 
   in terms of our chosen basis of the solution space.
\end{proof}

\begin{remark}
   In the proof \cref{propBCCoeffRels} we have introduced a set of (in)equalities
   governing the base change coefficients between two $\langle r, t \rangle$-strings
   $\sigma_x$ and $\sigma_z$ and shown that its of space of solutions is a 
   free $\Z[v,v^{-1}]$-module of rank $m-1$.
   Of course, not every solution gives a possible set of base change 
   coefficients $\{\pre{p}{m}_{z_j, x_i}  \; \vert \; 1 \leqslant 
   i, j \leqslant m-1\}$ as these coefficients have to satisfy more constraints:
   \cref{propPCanProps}(iii) shows that these coefficients are self-dual 
   and have non-negative integers as coefficients. Moreover, due to 
   \cref{propPCanProps}(v) these coefficients satisfy parity vanishing
   for fixed $i$ and arbitrary $1 \leqslant j \leqslant m-1$. This is also the 
   underlying reason why we could partition the set of inequalities in two sets. 
\end{remark}

\begin{cor}
   For $z \leqslant x \in \mathcal{D}_R(r, t)$ one has:
   \[ \pre{p}{m}_{z, x} = \pre{p}{m}_{z^{\ast}, x^{\ast}} \]
\end{cor}
\begin{proof}
   Note that $\pre{p}{m}_{z, x} = \pre{p}{m}_{z^{\ast}, x^{\ast}}$ asks only for
   $a_{i, j} = a_{m-i, m-j}$ (in the notation of the proof of \cref{propBCCoeffRels})
   for $1 \leqslant i, j \leqslant m-1$ whereas we have
   shown many more relations among these coefficients in the last proposition.
\end{proof}

In the proof of \cref{propBCCoeffRels} we have seen that when translating $\pre{k}{B}_x$
by $\pre{k}{B}_t$ for $x \in \mathcal{D}_R(r,t)$ with $xt > x$ the available coefficient 
of $\kl{z}$ in $\pcan{x} \kl{t}$ for $z \in \mathcal{D}_R(r,t)$ is completely subsumed by the 
neighbouring elements of $x$ in its right $\langle r, t \rangle$-string. This implies 
the following about the structure coefficients:

\begin{cor}
   \label{corVanishingStructCoeffsInStr}
   Let $x, z \in \mathcal{D}_R(r,t)$ with $xt > x$. Then $\pre{p}{\mu}_{x, t}^z$
   vanishes unless $x$ and $z$ are neighbouring elements in the same right 
   $\langle r, t \rangle$-string.
\end{cor}

The next result will allow us to apply the star-operations to the study of $p$-cells.
It is a generalization of \cite[(10.4.1), (10.4.2) and (10.4.3)]{LuCellsInAffWeylGrpsI}:

\begin{prop}
   \label{propStructCoeffRels}
   Let $\sigma_x = \{ x_1 < x_2 < \dots \}$ \textup{(}resp.  $\sigma_z = \{ z_1 < z_2 < \dots \}$\textup{)} 
   be two right $\langle r, t \rangle$-strings consisting of $m - 1$ elements. 
   For any $s \in S \setminus \desc{L}(x_1)$ all the relations stated in 
   \cref{propBCCoeffRels} hold with $\pre{p}{m}_{z_j, x_i}$ replaced by 
   $\pre{p}{\mu}_{s, x_i}^{z_j}$.
\end{prop}
\begin{proof}
   \cref{corRStringsInRCells} below shows that all elements in $\sigma_x$ lie
   in the same right $p$-cell. We therefore deduce from \cref{lemCleDesc} that 
   all elements in $\sigma_x$ have the same left descent set and satisfy 
   $s x_i > x_i$ for $1 \leqslant i \leqslant m-1$.
   
   Write $x_0$ (resp. $x_m$) for the shortest (resp. longest) element in the 
   right $\langle r, t \rangle$-coset of $x$. Similarly for $z_0$ and $z_m$.
   We are interested in the structure coefficients $\pre{p}{\mu}_{s, x_i}^{z_j}$
   for $1 \leqslant i, j \leqslant m-1$. To simplify notation, write 
   $a_{j, i} = \pre{p}{\mu}_{s, x_i}^{z_j}$.
   
   
   Fix $1 \leqslant i \leqslant m-1$ arbitrary. We may assume without loss
   of generality $x_i t > x_i$ and thus $r \in \desc{R}(x_i)$. The main idea is 
   to express $\pcan{s} \pcan{x_i} \pcan{t}$ in the $p$-canonical basis in 
   different ways and to analyze the coefficients
   in front of basis elements indexed by elements in $\sigma_z$. Depending
   on which multiplication in $\pcan{s} \pcan{x_i} \pcan{t}$ we first carry out, 
   we get two ways to express this product in terms of the $p$-canonical basis. On the one hand we have:
   \begin{align}
      \pcan{s} \pcan{x_i} \pcan{t} &= (\sum_{u \leqslant s x_i} \pre{p}{\mu}_{s, x_i}^u \pcan{u}) \pcan{t} \nonumber \\
         &= \sum_{\substack{u \leqslant sx_i \\ ut < u}}(v+v^{-1}) \pre{p}{\mu}_{s, x_i}^u \pcan{u}
         + \sum_{\substack{u \leqslant sx_i \wedge w \leqslant ut \\ ut > u}} \pre{p}{\mu}_{s, x_i}^u \pre{p}{\mu}_{u, t}^w \pcan{w}
   \label{eqnProd1}
   \end{align}
   Consider an element $u \leqslant sx_i$ with $\pre{p}{\mu}_{s, x_i}^u \ne 0$.   
   The version of \cref{lemCleDesc} for left cells implies $r \in \desc{R}(u)$ and
   in particular $u$ is not minimal in its right $\langle r, t\rangle$-coset. This 
   implies:
   \[ u \in \mathcal{D}_R(r, t) \Leftrightarrow ut > u \]
   Assume $u \in \mathcal{D}_R(r,t)$ (otherwise $\pcan{u}\pcan{t} = (v+v^{-1})\pcan{u}$
   does not contain any $p$-canonical basis elements indexed by elements in $\sigma_z$).
   Combining \cref{corVanishingStructCoeffsInStr} and \cref{corParStrCoeffs} 
   we have full control over the $p$-canonical basis 
   elements indexed by elements in $\mathcal{D}_R(r,t)$ that occur in 
   $\pcan{u} \pcan{t}$. Therefore, the only $p$-canonical basis elements that are
   indexed by elements in $\sigma_z$ and that occur with non-trivial coefficient
   in \eqref{eqnProd1} are the following:
   \[ \sum_{1 \leqslant j \leqslant m-1} a_{j, i}
      (\delta_{z_{j+1} \in \sigma_z} \pcan{z_{j+1}} + 
      \delta_{z_{j-1} \in \sigma_z} \pcan{z_{j-1}}) \]

   On the other hand we can rewrite the product as follows:
   \begin{align}
   \pcan{s} \pcan{x_i} \pcan{t} &= \pcan{s} (\sum_{u \leqslant x_i t} \pre{p}{\mu}_{x_i, t}^u \pcan{u}) \nonumber \\
    &= \sum_{\substack{w \leqslant x_it \\ sw < w}}(v+v^{-1}) \pre{p}{\mu}_{x_i,t}^w \pcan{w}
         + \sum_{\substack{u \leqslant x_i t \wedge w \leqslant su \\ su > u}} \pre{p}{\mu}_{x_i, t}^u \pre{p}{\mu}_{s, u}^w \pcan{w}
   \label{eqnProd2}
   \end{align}
   Consider an element $u \leqslant x_i t$ with $\pre{p}{\mu}_{x_i, t}^u \ne 0$.
   It follows that $u$ has $t$ in its right descent set. The version of 
   \cref{lemCleDesc} for left cells shows that $\pcan{s} \pcan{u}$ can only
   contribute $p$-canonical basis elements indexed by elements in $\mathcal{D}_R(r,t)$
   if $u$ lies in $\mathcal{D}_R(r,t)$ itself. (Observe that 
   \cref{corVanishingStructCoeffsInStr} together with the fact that elements in $\sigma_x$
   have the same left descent set implies that the sum $\sum_{\substack{w \leqslant x_it \\ sw < w}}(v+v^{-1}) \pre{p}{\mu}_{x_i,t}^w \pcan{w}$ cannot contribute any 
   $p$-canonical basis elements indexed by elements in $\sigma_z$.)
   Using \cref{corVanishingStructCoeffsInStr} and \cref{corParStrCoeffs} 
   again, we see that the only $p$-canonical basis elements that 
   are indexed by elements in $\sigma_z$ and that occur with non-trivial coefficient
   in \eqref{eqnProd2} are the following:
   \[ \sum_{1 \leqslant j \leqslant m-1} (\delta_{x_{i+1} \in \sigma_x} 
      a_{j, i+1} + \delta_{x_{i-1} \in \sigma_x} a_{j, i-1}) \pcan{z_j} \]
   Comparing coefficients in front of $\pcan{z_j}$ we get:
   \[ \delta_{z_{j-1} \in \sigma_z} a_{j-1, i} + 
      \delta_{z_{j+1} \in \sigma_z} a_{j+1, i} =
      \delta_{x_{i+1} \in \sigma_x} a_{j, i+1} + 
      \delta_{x_{i-1} \in \sigma_x} a_{j, i-1} \]

   Letting $i$ and $j$ vary, we see that the $a_{j, i}$'s satisfy precisely the inequalities 
   \eqref{eqnEven1} - \eqref{eqnEven13} and \eqref{eqnOdd1} - \eqref{eqnOdd12} (with equality).
\end{proof}

\begin{cor}
   \label{corStructCoeffsStar}
   Let $x, y \in \mathcal{D}_R(r,t)$ and $s \in S$ such that $sx > x$. Then we have
   \[ \pre{p}{\mu}_{s, x}^y = \pre{p}{\mu}_{s, x^{\ast}}^{y^{\ast}} \]
   where $(-)^{\ast}$ is the right star-operation with respect to $\{r, t\}$.
\end{cor}

Using the $\Z[v, v^{-1}]$-linear anti-involution $\iota$ on $\heck$ we can translate
all the results in this section about right strings and right star-operations
into results about left strings and left star-operations.

\todo[inline]{Understand the structure coefficients in $B_2$ for $p = 2$ 
and in $G_2$ for $p \in \{2, 3\}$ and check what holds in these cases!}

\subsection{Consequences for \texorpdfstring{$p$}{p}-Cells}
\label{secStrCells}

As we have seen in the last section one needs some assumptions on $p$ for 
the left and right star-operations to be well-behaved. Therefore, 
we keep these assumptions throughout this subsection. Fix for the rest of the 
section two simple reflections $r, t \in S$ with $3 \leqslant m \defeq m_{r, t} < \infty$.
Throughout the section we consider the right star-operation $(-)^{\ast}$ with respect to $\{r, t\}$.

Under these assumptions, the $p$-cells in any finite Weyl group of rank $2$ coincide
with the Kazhdan-Lusztig cells. Therefore, \cref{thmParaCombPCells} implies for 
$I = \{r, t\} \subseteq S$:

\begin{lem}
   \label{corRStringsInRCells}
   Let $\sigma$ be any right $\langle r, t\rangle$-string.
   Then all elements of $\sigma$ lie in the same right $p$-cell of $W$.
\end{lem}

In particular, we have the following result:
\begin{cor}
   \label{lemRStarPCell}
   For all $x \in \mathcal{D}_R(r,t)$, $x$ and $x^{\ast}$ lie in the same right $p$-cell.
\end{cor}

The following important result follows from \cref{corStructCoeffsStar}:
\begin{thm}
   \label{propLCellRStar}
   For $x, y \in \mathcal{D}_R(r, t)$ we have:
   \[ x \cle{L} y \Leftrightarrow x^{\ast} \cle{L} y^{\ast} \]
   In particular, if $x$ and $y$ lie the same left $p$-cell, then the same holds
   for $x^{\ast}$  and $y^{\ast}$.
\end{thm}

\begin{defn}
   For $r, t \in S$ with $rt \ne tr$ and $x \in \mathcal{D}_R(r,t)$ we denote by 
   $\sigma_x$ the right $\langle r, t \rangle$-string through $x$.
   Define $\mathfrak{T}_{r,t}(x) \defeq \{xr, xt\} \cap \mathcal{D}_R(r,t)$ to 
   be the neighbouring elements of $x$ in $\sigma_x$. 
   
   View $\mathfrak{T}_{r,t}$ as a map $\sigma_x \rightarrow \mathcal{P}(\sigma_x)$
   where $\mathcal{P}(\sigma_x)$ denotes the power set of $\sigma_x$.
   We define $\mathfrak{T}_{r,t}^2: \sigma_x \rightarrow \mathcal{P}(\sigma_x)$
   to be the map sending $y \in \sigma_x$ to $\bigcup_{z \in \mathfrak{T}_{r,t}(y)} 
   \mathfrak{T}_{r,t}(z)$. For $l\geqslant 2$, the map $\mathfrak{T}_{r,t}^l$ is defined 
   inductively in a similar fashion.
\end{defn}

Actually, one can characterize precisely the possible left $p$-cell preorder relations
among elements in right $\langle r, t \rangle$-strings:
\begin{prop}
   \label{propLCellPreorderRelsInRString}
   Let $\sigma = \{ x_1 < x_2 < \dots \}$ \textup{(}resp.  $\sigma' = \{ y_1 < y_2 < \dots \}$\textup{)} 
   be two right $\langle r, t \rangle$-strings consisting of $m-1$ elements. 
   Up to possibly interchanging the roles of $\sigma$ and $\sigma'$, the set of 
   left $p$-cell preorder relations between the elements of these
   two strings is one of the following:
   \begin{align}
      \text{no relation: } & \{ \} \tag{$\varnothing$} \label{eqnNoRel} \\
      \text{trivial case: } & \{ x_i \cle{L} y_i \; \vert \; \forall 1 
            \leqslant i \leqslant m-1 \} \tag{T} \label{eqnTrivCase} \\
      \text{permuted case: } & \{ x_i \cle{L} y_{\pi(i)} \; \vert \; \forall 1 
            \leqslant i \leqslant m-1 \} \tag{P} \label{eqnPermCase} \\
      \text{neighbour case: } 
      		&\{ x_i \cle{L} y \; \vert \; \forall 1 \leqslant i 
            \leqslant m-1, y \in \mathfrak{T}^k_{r,t}(y_i) \} 
            \tag{$N_k$} \label{eqnNCase} \\
      \substack{\text{permuted }\\\text{neighour case: }} 
      		&\{ x_i \cle{L} y_{\pi(j)} \; \vert \; \forall 1 \leqslant i 
            \leqslant m-1, y_j \in \mathfrak{T}^l_{r,t}(y_i) \} 
            \tag{$PN_l$} \label{eqnPNCase} \\
      \text{zig-zag case: } 
      		&\{ x_{i-1} \cle{L} y_{i+1}, x_i \cle{L} y_i, x_{i+1} \cle{L} y_{i-1} 
      			\; \vert \; \forall 2 \leqslant i \leqslant m-2 \} 
            \tag{$Z$} \label{eqnZCase}
   \end{align}
   for $1 \leqslant k \leqslant m-2$, $1 \leqslant l \leqslant m-4$ where $\pi = (1, m-1)(2,  m-2)$ is a permutation of $\{1, 2, \dots, m-1\}$.
   
\end{prop}
\begin{proof}

   The proof strategy is as follows: By applying \cref{propStructCoeffRels} to 
   the relation $y \multGe{L} x$ via $s \in S$, we get other elementary 
   left $p$-cell relations between the elements in $\sigma_x$ and $\sigma_y$. 
   The idea is to encode them in a function $f: \sigma_y \rightarrow \mathcal{P}(\sigma_x)$ 
   such that 
   \[ \pre{p}{\mu}_{s, y_i}^{x_j} \ne 0 \Leftrightarrow  x_j \in f(y_i) \text{.}\]
   The claim of the proposition is that there is a
   normal form of arbitrary finite compositions of such functions. If 
   $g: \sigma_x \rightarrow \mathcal{P}(\sigma_z)$ is another such function, their
   composition $g \circ f: \sigma_y \rightarrow \mathcal{P}(\sigma_z)$ sends
   $y_i$ to $\bigcup_{x_j \in f(y_i)} g(x_j)$. This is simply the composition
   of multi-valued functions. 
   
   In order to simplify notation, we will identify the $j$-th element in a right 
   $\langle r, t \rangle$-string with its position $j$. This allows us to view 
   any such $f$ as a map $\{1, 2, \dots, m-1\}  \rightarrow 
   \mathcal{P}(\{1, 2, \dots, m-1\})$. The composition of such functions 
   is to be understood in a similar fashion. One needs to keep track of the
   start and end string of the whole composition in order to retranslate 
   the function into the set of left $p$-cell relations. Using \cref{corStructCoeffsStar}
   we see that such a map is already fully determined by the images of
   $1,2, \dots, \lceil \frac{m-1}{2} \rceil$.
   
   We will prove the statement for $m = 6$ and leave the other cases to
   the reader. Apart from the identity $\id: i \mapsto \{i\}$ consider the following
   maps where we identify by slight abuse of notation the permutation $\pi$
   with the induced map $\{1, 2, \dots, 5\}  \rightarrow 
   \mathcal{P}(\{1, 2, \dots, 5\})$:
   \begin{alignat*}{4}
         &1 \mapsto \{3\}  &&1 \mapsto \{2\}\\
      a:\;&2 \mapsto \{2,4\} \qquad \qquad \qquad &b:\;&2 \mapsto \{1, 3\}\\
         &3 \mapsto \{1, 3, 5\} && 3 \mapsto \{2, 4\}\\
      &&&\\
         &1 \mapsto \{4\}  &&1 \mapsto\{5\}\\
      c:\;&2 \mapsto \{3, 5\} \qquad \qquad \qquad &\pi:\;&2 \mapsto \{4\}\\
         &3 \mapsto \{2, 4\} &&3 \mapsto \{3\}\\         
   \end{alignat*}
   Using \cref{propStructCoeffRels} it is easy to check that any elementary left
   $p$-cell relation implies relations encoded by one of the functions above, 
   for example $3 \multGe{L} 3$ implies relations encoded by either $\id$, $a$ or $\pi$.
   Analyzing the relations among compositions of these functions, one has:
   \begin{align*}
      a \circ \pi &=\pi \circ a = a\\
      b \circ \pi &= \pi \circ b = c\\
      b \circ c &= c \circ b\\
      \pi^2 &= \id\\
      a \circ b &= b \circ a = b ^ 3\\
      a^k = a^2 &= b^4 = b^{2k} = b^{2k} \circ \pi \text{ for } k \geqslant 2\\
      b^3 &= b^{2k+1} = b^{2k+1} \circ \pi \text{ for } k \geqslant 1
   \end{align*}
   Using these relations, we see that any finite composition 
   of these maps can be reduced to one of the following compositions
   \[ \id, \pi, b^k \text{ for } 1 \leqslant k \leqslant 4, b^l \circ \pi \text{ for } 
      1 \leqslant l \leqslant 2, a \text{.} \]
   These correspond precisely to the cases with at least one relation stated in 
   the proposition.
\end{proof}

\begin{remark}
   It should be noted that for $m < 6$ some of these cases coincide. For 
   example for $m=3$, the permutation $\pi$ is trivial and there are
   only three distinct cases: The permuted case coincides with the
   trivial case. Moreover, the zig-zag case does not contain any relations
   and the permuted neighbour case does not exist. 
   For $m=4$, there are four distinct cases: The zig-zag case reduces to the 
   permuted case and the permuted neighbour case does not exist.
\end{remark}

The normal forms given in the proof of \cref{propLCellPreorderRelsInRString} allow us
to deduce the following equivalences, which show how rigid the combinatorics in this
situation are. The reader should compare these equivalences with \cref{propStructCoeffRels}
which only deals with the generating relations for the $p$-cell preorder. It is a
generalization of \cite[Proposition 4.6 and Remark 4.7]{ShLeftCellsInAffWeylGrps}:

\begin{cor}
   Let $\sigma = \{ x_1 < x_2 < \dots \}$ \textup{(}resp.  $\sigma' = \{ y_1 < y_2 < \dots \}$\textup{)}
   be two right $\langle r, t \rangle$-strings consisting of $m-1$ elements. In addition
   to $x \cle{L} y \Leftrightarrow x^{\ast} \cle{L} y^{\ast}$ for $x \in \sigma$ and
   $y \in \sigma'$ we get the following equivalences:
   \begin{align}
   m = 4 &\Rightarrow \begin{cases}
                        x_1 \cle{L} y_2  \Leftrightarrow x_2 \cle{L} y_1 \\
                        x_2 \cle{L} y_2 \Leftrightarrow (x_1 \cle{L} y_1 \vee x_1 \cle{L} y_3)
                        \end{cases} \\
   m = 6 &\Rightarrow \begin{cases}
                        x_1 \cle{L} y_2  \Leftrightarrow x_2 \cle{L} y_1 \\
                        x_1 \cle{L} y_3  \Leftrightarrow x_3 \cle{L} y_1 \\
                        x_1 \cle{L} y_4  \Leftrightarrow x_4 \cle{L} y_1 \\
                        x_2 \cle{L} y_2  \Leftrightarrow (x_1 \cle{L} y_1 \vee x_1 \cle{L} y_3)\\
                        x_2 \cle{L} y_3  \Leftrightarrow x_3 \cle{L} y_2  \Leftrightarrow 
                                          (x_1 \cle{L} y_2 \vee x_1 \cle{L} y_4)\\
                        x_2 \cle{L} y_4  \Leftrightarrow (x_1 \cle{L} y_3 \vee x_1 \cle{L} y_5)\\
                        x_3 \cle{L} y_3  \Leftrightarrow (x_1 \cle{L} y_1 \vee x_1 \cle{L} y_3 \vee x_1 \cle{L} y_5)
                      \end{cases}
   \end{align}
\end{cor}

We can now generalize \cite[Proposition 10.7]{LuCellsInAffWeylGrpsI} as follows:
\begin{prop}
   \label{propCellStar}
   Let $r, t \in S$ and $\Gamma$ be a union of left $p$-cells such that 
   $\Gamma \subseteq \mathcal{D}_R(r,t)$.
   Then the following holds:
   \begin{enumerate}
    \item $\widetilde{\Gamma} \defeq (\bigcup_{w \in \Gamma} \sigma_w) \setminus \Gamma$
          is a union of left $p$-cells.
    \item If $\Gamma$ is a left $p$-cell, then $\widetilde{\Gamma}$ is a union of at
          most $m - 2$ left $p$-cells.
    \item If $\Gamma$ is a left $p$-cell, then $\Gamma^{\ast} \defeq \{ w^{\ast} \; 
          \vert \; w \in \Gamma \}$ is a left $p$-cell as well.
   \end{enumerate}
\end{prop}
\begin{proof}
   For two left $p$-cells $\Gamma_1$ and $\Gamma_2$ contained in $\mathcal{D}_R(r,t)$ 
   we have:
   \begin{align*}
      \widetilde{(\Gamma_1 \cup \Gamma_2)} &= \left( \bigcup_{w \in \Gamma_1 \cup \Gamma_2} 
         \sigma_w \right) \setminus (\Gamma_1 \cup \Gamma_2) \\
         &= \widetilde{\Gamma_1} \setminus \Gamma_2 \cup 
            \widetilde{\Gamma_2} \setminus \Gamma_1
   \end{align*}
   Therefore, we may assume without loss of generality that $\Gamma$ is a left $p$-cell. 
   It is enough to prove that $\bigcup_{w \in \Gamma} \sigma_w$ is a union of 
   left $p$-cells.
   
   Let $x \in \bigcup_{w \in \Gamma} \sigma_w$ and $y \in W$ such that $x \ceq{L} y$.
   Then there exists a sequence
   \[ P: x = x_0 \multGe{L} x_1 \multGe{L} \dots \multGe{L} x_k = y \multGe{L} x_{k+1} \multGe{L} 
   \dots \multGe{L} x_l = x \]
   of elements in $\mathcal{D}_R(r, t)$ (by \cref{lemCleDesc}) that 
   all lie in the same left $p$-cell as $x$ and $y$. Consider the 
   right $\langle r, t \rangle$-strings of 
   all the elements in the sequence and note that $\sigma_x$ contains an 
   element $\bar{x} \in \Gamma$. Since each $\langle r, t \rangle$-string
   is contained in a right $p$-cell (see \cref{corRStringsInRCells}), 
   the elements have the same left descent set.
   Thus for $0 \leqslant i < l$ all the elements $z \in \sigma_{x_i}$
   satisfy $s_i z > z$ where $s_i$ is the simple reflection used to
   get from $x_i$ to $x_{i+1}$. Express $\kl{s_i} \pcan{z}$ in the 
   $p$-canonical basis and remember the elements in $\sigma_{x_{i+1}}$
   indexing $p$-canonical basis elements with non-trivial coefficients.
   This gives the elementary left $p$-cell relations between $\sigma_{x_i}$
   and $\sigma_{x_{i+1}}$ induced by $s_i$. By combining the various
   elementary left $p$-cell relations between $\sigma_{x_i}$
   and $\sigma_{x_{i+1}}$ induced by $s_i$ along the sequence $P$,
   we obtain for any $z \in \sigma_x$ a possibly empty set $f_P(z)$ 
   of elements in $\sigma_x$ such that $z \cge{L} z'$ if and only 
   if $z' \in f_P(z)$. The proof of 
   \cref{propLCellPreorderRelsInRString} shows that for 
   $n \gg 1$ the image of $f_P^{2n}$ stabilizes in the sense that
   $z \in f_P^{2n}(z)$. The element $\bar{x}$ shows that there is an element in 
   $\sigma_y \cap \Gamma$ which implies $y \in \bigcup_{w \in \Gamma} \sigma_w$ 
   and finishes the proof of (i).
   
   We claim that any left $p$-cell $\Gamma'$ in $\bigcup_{w \in \Gamma} \sigma_w$
   intersects $\sigma_x$ non-trivially. Let $y \in \Gamma'$ and 
   $\bar{y} \in \sigma_y \cap \Gamma$. Use a sequence $P$ as above relating
   $\bar{x}$ and $\bar{y}$ to construct 
   $f_P: \sigma_x \rightarrow \mathcal{P}(\sigma_x)$. \cref{propLCellPreorderRelsInRString}
   shows that there exist $x' \in \sigma_x$ such that $x' \cge{L} y$. 
   Using the stabilization of $f_P$ as above, we get $x' \in f_P^{2n}(x')$ 
   for $n \gg 1$ and in particular $x'$ lies in $\Gamma' \cap \sigma_x$.
   This finishes the proof of the claim and shows that $\widetilde{\Gamma}$ is a
   union of at most $m - 2$ left $p$-cells giving (ii).
   
   (iii) is an immediate consequence of \cref{propLCellRStar}.
\end{proof}

\begin{remark}
   \label{remNumPCells}
   Analyzing the situation more carefully allows us to say more
   about the number of left $p$-cells in $\widetilde{\Gamma}$. Let 
   $x, y \in \Gamma$.  Consider $f_P^n$ constructed as in the last proof
   for a sequence $P$ relating $x$ and $y$.    
   The proof of \cref{propLCellPreorderRelsInRString} shows that 
   $f_P^{2n}$ for $n \gg 1$ stabilizes to one of the following maps:
   \begin{alignat*}{5}
      &f_{\text{triv}}: \sigma_x &&\longrightarrow \mathcal{P}(\sigma_x) \quad \text{or} \quad 
         &&f_{\text{nontriv}}: \sigma_x &&\longrightarrow \mathcal{P}(\sigma_x) \\
         &x_i &&\longmapsto \{x_i\} 
         &&x_{2l-1}&&\longmapsto \{x_{2k-1} \; \vert \; 
         1 \leqslant k \leqslant \floor*{\frac{m}{2}} \}\\
      &&&
         &&x_{2l} &&\longmapsto \{x_{2k} \; \vert \; 
         1 \leqslant k \leqslant \floor*{\frac{m-2}{2}}\}
   \end{alignat*}
   Finally, let $P$ vary over all potential sequences and $y$ vary over 
   the elements in $\Gamma$. If there exists an element $y \in \Gamma$
   and a sequence $P$ relating $x$ and $y$ such that $f_P^n$ stabilizes 
   to $f_{\text{nontriv}}$, then $\widetilde{\Gamma}$ is a left $p$-cell
   (note that $\widetilde{\Gamma}$ is always non-empty as all elements in
   $\Gamma$ have the same right descent set and this is not the case for all
   elements in $\sigma_x$). \todo[inline]{Find necessary and sufficient conditions for
   this to hold! Possible: $\Gamma$ contains elements whose positions in their strings
   do not sum up to $m$? $\Gamma$ contains more than one odd-numbered or more than
   one even-numbered element in a string?}
   
   Another situation, in which we can say more about the number of $p$-cells
   in $\widetilde{\Gamma}$ is the following:
   If $m \geqslant 4$ and there exists a string $\sigma$ such that 
   $\sigma \cap \Gamma$ contains only one odd-numbered element of $\sigma$, 
   then $\widetilde{\Gamma}$ decomposes into at least $2$ left $p$-cells 
   for descent set reasons. Therefore, under these assumptions $\widetilde{\Gamma}$
   contains precisely two right $p$-cells if $m = 4$.
\end{remark}

The definition of a $W$-graph from \cite[\S1]{KL} describes a based representation
of the Hecke algebra. A typical example is given by a Kazhdan-Lusztig cell
module equipped with the Kazhdan-Lusztig basis, for which the $W$-graph describes
the action of the generating set $\{ \std{s} \; \vert \; s  \in S \}$ of the 
Hecke algebra. In order to allow $W$-graphs to also describe $p$-cell modules, 
we need to slightly generalize the original definition, in which we use the generating
set $\{ \kl{s} \; \vert \; s \in S\}$ of the Hecke algebra instead:
\begin{defn}
   \label{defpWGraph}
   A coloured $W$-graph is a directed graph with vertices $\mathcal{V}$ and edges $\mathcal{E}$
   together with the following decorations:
   \begin{itemize}
      \item for each vertex $x \in \mathcal{V}$ a subset $I_x$ of $S$ \text{,}
      \item for each edge $(x, y) \in \mathcal{E}$ a family of Laurent polynomials
            \[\{ \mu_{s, x}^y \; \vert \; s \in I_y \setminus I_x \} \subset \Z[v, v^{-1}]\]
   \end{itemize}
   subject to the conditions below. Let $V$ be the free $\Z[v, v^{-1}]$-module
   with basis $\mathcal{V}$. For $s \in S$ define a $\Z[v, v^{-1}]$-linear endomorphism
   of $V$ as follows
   \[ \tau_s(x) = 
      \begin{cases}
         (v + v^{-1}) x & \text{if } s \in I_x \text{,} \\
         \sum_{\substack{(x, y) \in \mathcal{E} \\ s \in I_y}} \mu_{s, x}^y y 
         & \text{otherwise.}
      \end{cases}
   \]
   where the sum is finite due to the second condition below. Then a coloured $W$-graph
   is required to satisfy:
   \begin{enumerate}
      \item the morphism $\heck \rightarrow \End(V)$, $\kl{s} \mapsto \tau_s$
            extends to a morphism of $\Z[v, v^{-1}]$-algebras,
      \item for each pair $(x, s) \in \mathcal{V} \times S$ there are only finitely
            many edges $(x, y) \in \mathcal{E}$ with $\mu_{s, x}^y \ne 0$.
   \end{enumerate}
\end{defn}

It is immediate that we can associate to each left $p$-cell $C$ a coloured $W$-graph $\Gamma_C$ 
as defined above with $C$ as vertex set. For each $x \in C$ we set $I_x \defeq \desc{L}(x)$
and use edges to encode the structure coefficients for the $p$-canonical
basis $\pre{p}{\mu}_{s, x}^y$ for $x, y \in C$. Observe that right star operations
do not modify the left descent set. This combined with \cref{propCellStar} (iii) and 
\cref{corStructCoeffsStar} implies the following result:

\begin{lem}
   \label{lemCellModStar}
   Let $C$ be a left $p$-cell contained in $\mathcal{D}_R(r, t)$. Then
   the left $p$-cell module associated to $C$ is isomorphic to the left $p$-cell
   module associated to $C^{\ast}$ where $\ast$ is the right star-operation associated
   to $r$ and $t$. In particular, the coloured $W$-graphs $\Gamma_C$ and $\Gamma_{C^{\ast}}$
   are isomorphic as decorated, directed graphs.
\end{lem}

\subsection{Vogan's Generalized \texorpdfstring{$\tau$}{Tau}-Invariant}

Vogan defined in \cite[Definition 3.10]{VoGenTauInv} an invariant of Kazhdan-Lusztig 
left cells in the setting of primitive ideals for semi-simple Lie algebras. This became
known  as Vogan's generalized $\tau$-invariant and was generalized in 
\cite[Definition 5.1]{BGTauInv} to arbitrary Coxeter groups. (Observe that only pairs
of simple reflections $\{r, t\} \in W$ with $m_{r,t} \in \{3, 4\}$ matter 
in the following definition.)
\todo[inline]{Generalize to $m_{r,t} = 6$?! No real advantage! Only for $\widetilde{G}_2$.}
\begin{defn}
   \label{defGenTauInv1}
   Since $\mathfrak{T}_{r,t}(x)$ for $r, t \in S$ consists of one or two elements,
   we use the following convention: We consider $\mathfrak{T}_{r,t}(x)$ as a multiset 
   with two identical elements in the following if $\{xr, xt\} \cap \mathcal{D}_R(r,t)$ 
   is of cardinality $1$.
   
   We define a sequence of equivalence relations $\approx_n$ for $n \in \N$ as follows.
   For $x, y \in W$ we write:
   \begin{alignat*}{3}
      x \approx_0 y &\text{ if } &&\desc{R}(x) = \desc{R}(y)\text{,} \\
      x \approx_{n+1} y &\text{ if } &&x \approx_n y  \text{ and for any pair } 
         r, t \in S \text{ such that } m_{r,t} \in \{3, 4\} \text{ and } \\
     &&&x, y \in \mathcal{D}_R(r,t) \text{ with } \mathfrak{T}_{r,t}(x) = \{x_1, x_2\} 
         \text{ and } \mathfrak{T}_{r,t}(y) = \{y_1, y_2\}\\
     &&&\text{we have: } x_1 \approx_n y_1 \text{, } x_2 \approx_n y_2 \text{ or } x_1 \approx_n y_2\text{, }
     x_2 \approx_n y_1
   \end{alignat*}
   We say that $x$ and $y$ have the same \emph{generalized $\tau$-invariant} if
   $x \approx_n y$ holds for all $n \geqslant 0$. We call the set 
   \[\{ w \in W \; \vert \; x \approx_n w \text{ for all } n \geqslant 0 \}\]
   the $\tau$-equivalence class of $x$.
\end{defn}

\begin{remark}
   Observe that the last definition admits an obvious generalization allowing the
   case $m_{r, t} = 6$. For our current applications, we do not need this
   generality. Thus we exclude this case just like Vogan did in the original
   definition.
\end{remark}

The following result shows that left $p$-cells give a refinement of the $\tau$-equivalence classes.
It is a generalization of \cite[Theorem 5.2]{BGTauInv}:
\begin{thm}
   Assume $p > 2$ if there exist $s \ne t \in S$ with $m_{s, t} = 4$. Let $\Gamma$ be a left $p$-cell. 
   Then all elements in $\Gamma$ have the same generalized $\tau$-invariant.
   In particular, any $\tau$-equivalence class decomposes into left $p$-cells.
\end{thm}
\begin{proof}
   The proof of \cite[Theorem 5.2]{BGTauInv} works after replacing all
   Kazhdan-Lusztig related constructions by their $p$-canonical analogues 
   with the following modifications:
   First add a zeroth case in which there exists a right $\langle s, t\rangle$-string $\sigma$
   such that $\Gamma \cap \sigma = \{x's, x'sts\}$ for some minimal element $x'$ in its right
   $\langle s, t\rangle$-coset. In this case, $\widetilde{\Gamma}$ consists of a single 
   left $p$-cell, which allows us to conclude.
   Finally in the second case, use the left $p$-cell preorder relations for the roles 
   of $y$ and $w$ swapped instead of appealing to \cite[Corollary 6.3]{LuCellsInAffWeylGrpsI}. 
\end{proof}

\begin{cor}
   Assume $p > 2$ if there exist $s \ne t \in S$ with $m_{s, t} = 4$. Assume that two elements in $W$
   have the same generalized $\tau$-invariant if and only if they belong
   to the same Kazhdan-Lusztig left cell (i.e. the generalized $\tau$-invariant
   gives a complete invariant of Kazhdan-Lusztig left cells).
   Then Kazhdan-Lusztig left cells decompose into left $p$-cells.
\end{cor}

In \cite[Theorem 6.5]{VoGenTauInv}, Vogan shows that the generalized $\tau$-invariant
gives a complete invariant in finite type $A$. The same holds in finite types $B/C$
(see \cite[Theorem 3.5.9]{GaPrimIdealsIII} based on \cite[Definitions 2.1.3. - 2.1.7., 3.2.1. and 3.4.1.]
{GaPrimIdealsII, GaPrimIdealsIII}) and finite type $E_6$ (see \cite{ToLCellsE6}). 
Therefore, we have:

\begin{cor}
   \label{corLCellsDecomp}
   The Kazhdan-Lusztig left cells in finite type $A$ decompose into left $p$-cells.
   The same holds in finite types $B$ and $C$ for $p > 2$ and in finite type $E_6$.
\end{cor}

Even though we currently can only calculate the full $p$-canonical basis in types
$B_n$ and $C_n$ for $n \leqslant 5$ and in these groups only $2$-torsion occurs, the last
result is in particular of interest due to \cite{WTorsion, WGeomTorsion}. In these papers
Williamson shows that there is torsion in the local integral intersection
cohomology of Schubert varieties in the flag variety of the general linear group
that grows exponentially in the rank. This implies that the $p$-canonical basis
does not coincide with the Kazhdan-Lusztig basis for arbitrarily large primes
in type $A_n$ (and that the primes for which this occurs grow exponentially in
the rank). Since type $A_n$ embeds into type $B_{n+1}$ and $C_{n+1}$ this gives 
many interesting examples for which the $p$-canonical basis and the Kazhdan-Lusztig 
basis differ for large primes $p$.

In \cite[Remark following Proposition 4.4]{VoGenTauInv} Vogan mentions that the generalized
$\tau$-invariant is not a complete invariant in type $F_4$. Similarly in type $D_n$ for 
$n \geqslant 6$ (see the introduction of \cite{GaPrimIdealsI}).

As proposed by \cite[\S4]{VoOrdPrimSpec} and \cite[Remark 5.3]{BGTauInv} we could 
have defined an invariant as follows:
\begin{defn}
   \label{defGenTauInv2}
   
   As in \cref{defGenTauInv2} we define a sequence of equivalence relations 
   $\approx_n'$ for $n \in \N$ as follows. For $x, y \in W$ we write:
   \begin{alignat*}{3}
      x \approx_0' y &\text{ if } &&\desc{R}(x) = \desc{R}(y)\text{,} \\
      x \approx_{n+1}' y &\text{ if } &&x \approx_n' y  \text{ and for any pair } 
         r, t \in S \text{ such that } \infty > m_{r,t} \geqslant 3 \text{ and}\\
     &&&x, y \in \mathcal{D}_R(r,t) \text{ we have: } x^{\ast} \approx_n' y^{\ast}
   \end{alignat*}
   where $(-)^{\ast}$ is the right star-operation with respect to $\{r, t\}$.
   As above we say that $x$ and $y$ have the same \emph{generalized $\widetilde{\tau}$-invariant} if
   $x \approx_n' y$ holds for all $n \geqslant 0$. We call the set 
   \[\{ w \in W \; \vert \; x \approx_n' w \text{ for all } n \geqslant 0 \}\]
   the $\widetilde{\tau}$-equivalence class of $x$.
\end{defn}

In this case, \cref{propLCellRStar} immediately implies that the partition of $W$ 
into left $p$-cells gives a refinement of the $\widetilde{\tau}$-equivalence classes:

\begin{cor}
   Assume $p > 2$ (resp. $p > 3$) if there exist $s \ne t \in S$ with $m_{s, t} = 4$
   (resp. $m_{s, t} = 6$). Let $\Gamma$ be a left $p$-cell. 
   Then all elements in $\Gamma$ have the same generalized $\widetilde{\tau}$-invariant.
   In particular, any $\widetilde{\tau}$-equivalence class decomposes into left $p$-cells.
\end{cor}

In \cite[Conjecture 6.9]{GHCellsInE8} Geck and Halls propose a slight variation
of Vogan's original conjecture (see \cite[Conjecture 3.11]{VoGenTauInv}) which goes as follows:

\begin{conj}
   For any finite Coxeter group $W_f$ two elements $x, y \in W_f$ belong to the
   same Kazhdan-Lusztig left cell if and only if $x$ and $y$ lie in the same Kazhdan-Lusztig
   two-sided cell and in the same $\widetilde{\tau}$-equivalence class.
\end{conj}

They verified their conjecture in all types $BC_n$ and $D_n$ for $n \leqslant 9$ 
and in all exceptional types. Moreover, they mention that in type $F_4$ the Kazhdan-Lusztig
left cells are precisely the $\widetilde{\tau}$-equivalence classes. From this, 
we deduce that for $p \geqslant 3$ the Kazhdan-Lusztig left cells in type $F_4$
decompose into left $p$-cells. Our explicit computer calculations show that the $p$-canonical
basis and the Kazhdan-Lusztig basis in type $F_4$ only differ for $p \in \{2,3\}$ and 
for  $p=2$ the Kazhdan-Lusztig left cells do not decompose into left $p$-cells.

\subsection{\texorpdfstring{$p$}{p}-Cells in Type \texorpdfstring{$A$}{A}}

Kazhdan-Lusztig cells in type $A$ can be characterized using the Robinson-Schensted
correspondence. This result is usually attributed to \cite[\S4]{KL} even
though the result is not stated in this form and depends on results of Joseph
and Vogan in the setting of primitive ideals. The first combinatorial proof is
due to \cite{GMCellsTypeA} and \cite{ArRSLCells}. In this section we transfer 
almost verbatim Ariki's proof to the modular setting. Since we feel that the proof is not
as well documented as it should be, we decided to give the proof here.

Throughout this section we assume that we used the Cartan matrix in finite type $A_{n-1}$
as input. In this case $(W, S)$ can be identified with $(S_n, \{s_1, \dots, s_{n-1}\})$ 
the symmetric group together with the set of simple transpositions. Letting $S_n$ 
act on $\{1, 2, \dots, n\}$ on the left, we can write any element $w \in S_n$ 
uniquely as $w = w(1) w(2) \dots w(n)$ which we call string notation.

Recall the definition of the elementary Knuth transformation $K_i$ for $1 < i < n$: 
Let $x= x_1 x_2 \dots x_n, y = y_1 y_2 \dots y_n \in S_n$ be two elements of the symmetric group 
in string notation. Write $x \underset{K_i}{\approx} y$ if $x$ and $y$ differ only
on the substrings $x_{i-1} x_i x_{i+1}$ and $y_{i-1} y_i y_{i+1}$ and these substrings 
are related to each other in either of the following two ways
\[ bca \leftrightarrow bac \text{ or } cab \leftrightarrow acb \]
where $a < b < c$. We say that $x$ and $y$ are Knuth equivalent if they are related
by a sequence of elementary Knuth transformations $K_i$ for various $i$. The following 
result follows immediately from the definitions:

\begin{lem}
   \label{lemKEqStarOp}
   $\mathcal{D}_R(s_{i-1}, s_i)$ is the subset of elements in $S_n$ to which $K_i$
   can be applied. Moreover, $w \underset{K_i}{\approx} w^{\ast}$ for all 
   $w \in \mathcal{D}_R(s_{i-1}, s_i)$ where $(-)^{\ast}$ is the right 
   star-operation defined with respect to $\{s_{i-1}, s_i\}$.
\end{lem}

The Robinson-Schensted correspondence (see \cite[\S A.3.3]{BBCoxGrps} or \cite[\S4.1]{FuTableaux}) 
gives a bijection between the symmetric group $S_n$ and pairs of standard tableaux
of the same shape with $n$ boxes. The row-bumping algorithm gives a way to explicitly calculate 
the image $(P(w), Q(w))$ of $w \in S_n$ under the Robinson-Schensted correspondence.
We will need the following important classical result about the Robinson-Schensted
correspondence (see \cite[\S4.1, Corollary to Symmetry Theorem]{FuTableaux}):

\begin{thmlab}[Symmetry Theorem for $S_n$]
   \label{thmSymm}
   If $w \in S_n$ corresponds to $(P(w), Q(w))$, then $w^{-1}$ corresponds to
   $(Q(w), P(w))$ under the Robinson-Schensted correspondence.
\end{thmlab}

Moreover, we need the following result by Knuth (see \cite[Theorem 6]{Kn}):
\begin{thm}
   \label{thmKEqPSymbols}
   Let $x, y \in S_n$. Then $x$ and $y$ are Knuth equivalent if and only if
   $P(x) = P(y)$.
\end{thm}

The goal is to prove the following theorem which is known for Kazhdan-Lusztig cells
(see \cite[Theorem 5.4.1 and Corollary 5.4.2]{W1} or \cite[Chapter 6, Exercise 11]{BBCoxGrps}):
\begin{thm}
   \label{thmPCellsTypeA}
   For $x, y \in S_n$ we have:
   \begin{align*}
      x \ceq{L} y &\Leftrightarrow Q(x) = Q(y) \\
      x \ceq{R} y &\Leftrightarrow P(x) = P(y) \\
      x \ceq{LR} y &\Leftrightarrow Q(x) \text{ and } Q(y) \text{ have the same shape}
   \end{align*}
   In particular, Kazhdan-Lusztig cells and $p$-cells of $S_n$ coincide.
\end{thm}
\begin{remark}
   To prove the characterization of the left $p$-cell cells we only need
   the following ingredients: For the direction $\Leftarrow$ we need \cref{corLeftRightEq}
   and \cref{lemRStarPCell}. The other direction is more involved and relies on
   the version of \cref{lemCleDesc} for left cells and crucially on \cref{propLCellRStar}.
\end{remark}
\begin{proof}
   We will first deal with the statement about left $p$-cells:
   \begin{description}[leftmargin=0pt]
   \item[$\Leftarrow$] By \cref{thmSymm} we have $P(x^{-1}) = Q(x) = Q(y) = P(y^{-1})$ which implies
         by \cref{thmKEqPSymbols} that $x^{-1}$ and $y^{-1}$ are related by a sequence
         of elementary Knuth transformations (i.e. right start operations with respect
         to different subsets of $S$ consisting of two neighbouring simple reflections).
         Successive applications of \cref{lemRStarPCell} show that $x^{-1} \ceq{R} y^{-1}$.
         This is equivalent to $x \ceq{L} y$ by \cref{corLeftRightEq}.
   \item[$\Rightarrow$] Denote the shape of $Q(x)$ (resp. $Q(y)$) by $\pi_x$ 
         (resp. $\pi_y$) and let $P_{x}$ (resp. $P_{y}$) be the column superstandard tableau
         (see \cite[\S A3.5]{BBCoxGrps} for the definition)
         of shape $\pi_x$ (resp. $\pi_y$). The Robinson-Schensted correspondence 
         gives elements $\hat{x}, \hat{y} \in S_n$ with $P$ and $Q$-symbols
         $(P_x, Q(x))$ and $(P_y, Q(y))$ respectively. The implication we proved above gives
         $x \ceq{L} \hat{x}$ and $y \ceq{L} \hat{y}$ which implies by our assumption
         $\hat{x} \ceq{L} \hat{y}$. In order to show $Q(x) = Q(y)$ 
         consider the elements $x', y'' \in S_n$ corresponding to $(P_x, P_x)$
         and $(P_y, P_y)$ respectively (under the Robinson-Schensted correspondence).
         \cref{thmKEqPSymbols} implies that the elements $\hat{x}$ and $x'$ as well
         as $\hat{y}$ and $y''$ are related by sequences of Knuth moves:
         \begin{align*}
            x' &= K_{i_r} \circ \dots \circ K_{i_1}(\hat{x}) \\
            y'' &= K_{j_s} \circ \dots \circ K_{j_1}(\hat{y})
         \end{align*}
         \cref{lemCleDesc} for the left $p$-cell preorder shows that the elements 
         $\hat{x}$ and $\hat{y}$ have the same right descent set, so the same
         right star-operations or Knuth moves (see \cref{lemKEqStarOp}) can be applied
         to both elements. By \cref{propLCellRStar} we have 
         $K_{i_1} (\hat{x}) \ceq{L} K_{i_1}(\hat{y})$ and $K_{j_1} (\hat{x}) \ceq{L} K_{j_1}(\hat{y})$. 
         Therefore, we can repeat the argument to see that the following elements are well-defined:
         \begin{align*}
            x'' &= K_{j_s} \circ \dots \circ K_{j_1}(\hat{x}) \\
            y' &= K_{i_r} \circ \dots \circ K_{i_1}(\hat{y})
         \end{align*}
         Moreover, we have $x' \ceq{L} y'$ as well as $x'' \ceq{L} y''$ and thus
         $\desc{R}(x') = \desc{R}(y')$ and $\desc{R}(x'') = \desc{R}(y'')$. Using 
         \cref{thmKEqPSymbols} we see that $P(x'') = P(\hat{x}) = P_x$ and 
         $P(y') = P(\hat{y}) = P_y$.
         
         Denote the column lengths of $\pi_x$ (resp. $\pi_y$) by $l_1, l_2, \dots$
         (resp. $k_1, k_2, \dots$). It follows from the row-bumping algorithm that
         $x'$ is the longest element in the parabolic subgroup 
         $S_{l_1} \times S_{l_2} \times \dots$ of $S_n$, i.e. in string notation the element
         \[ l_1, l_1 - 1, \dots , 1, l_1+ l_2, l_1 + l_2 - 1, \dots, l_1+1, \dots \text{.} \]
         From $\desc{R}(x') = \desc{R}(y')$ and the characterization of right descent sets
         for elements in $S_n$ in terms of inversions, we deduce that in the string
         notation for $y'$ the first $l_1$ letters are decreasing as well as the next
         $l_2$ letters and so on. Similarly, we may use that $y''$ is the longest element
         in the parabolic subgroup $S_{k_1} \times S_{k_2} \times \dots$ of $S_n$
         and $\desc{R}(x'') = \desc{R}(y'')$ to deduce that in the string notation
         of $x''$ the first $k_1$ letters are decreasing, the next $k_2$ letters are
         in decreasing order, etc.
         
         Applying the row-bumping algorithm to $y'$ to calculate $P(y') = P_y$, we 
         obtain the inequality $k_1 \geqslant l_1$. Using $x''$ instead,
         we get the opposite inequality giving $l_1 = k_1$. Moreover,
         this shows that when inserting the next $l_2$ letters of $y'$, no row
         bumping occurs in the first column (otherwise we would have $k_1 > l_1$)
         and thus we have $k_2 \geqslant l_2$. Again, we may use $x''$ to
         get the opposite inequality and to show $k_2 = l_2$. Repeating the argument,
         we get $\pi_x = \pi_y$ and thus $Q(y') = P_y = P_x = Q(x'')$ (by the definition
         of the column superstandard tableau and the fact that $Q(-)$ encodes
         the order in which boxes are added in the course of the row-bumping
         algorithm). This shows $y' = x' = y'' = x''$ as well as $\hat{x} = \hat{y}$
         by unravelling the sequences of Knuth moves. Finally, the Robinson-Schensted
         correspondence gives $Q(x) = Q(y)$ and finishes the proof of the characterization
         of left $p$-cells in terms of $Q$-symbols.
   \end{description}
   
   Using \cref{thmSymm} and \cref{corLeftRightEq} we obtain the version for right 
   $p$-cells. Finally, we prove the statement about two-sided $p$-cells:
   
   \begin{description}[leftmargin=0pt]
    \item[$\Leftarrow$] \cref{thmKEqPSymbols} shows that any two elements of $S_n$ with 
   the same $P$-symbol can be related by a sequence of elementary Knuth transformations. Dually, any two
   elements with the same $Q$-symbol are linked by a sequence of elementary dual Knuth
   transformations, which we did not introduce, but which correspond to left star-operations.
   Given an element $x \in S_n$ we can thus transform its $P$ and $Q$ symbols using
   Knuth transformations and their duals into any pair of given standard tableaux of
   the same shape. Denote by $\pi$ the shape of $Q(x)$ and let $P_{\pi}$ be the 
   column superstandard tableau of shape $\pi$. The statement about left and
   right $p$-cells shows that $x$ lies in the same two-sided $p$-cell as the 
   element $w_{\pi}$ corresponding to $(P_{\pi}, P_{\pi})$ under the Robinson-Schensted 
   correspondence. From this, the reader easily deduces the direction $\Leftarrow$.
   
   \item[$\Rightarrow$] Given an element $x \in S_n$, denote by $\pi$ the shape
   of $Q(x)$. Let $w_{\pi}$ be as defined above. We claim that $y \cle{LR} x$
   implies $y \cle[0]{LR} x$.
   
   Note that $w_{\pi}$ is the longest element in a standard parabolic subgroup 
   of $S_n$. Thus we have 
   \begin{equation}
      \label{eqPCanEqCanSn}
      \pcan{w_{\pi}} = \kl{w_{\pi}}\text{.} 
   \end{equation}
   As $Q(w_{\pi})$ and $Q(x)$ have the same shape, the direction $\Leftarrow$ implies
   $w_{\pi} \ceq{LR} x$ and $w_{\pi} \ceq[0]{LR} x$. The relation $x \cle{LR} w_{\pi}$
   together with (\ref{eqPCanEqCanSn}) gives us for all $z \leqslant x$:
   \[ \pre{p}{m}_{z, x} \ne 0 \Rightarrow z \cle[0]{LR} w_{\pi} \]
   Therefore, any element $y \cle{LR} x$ satisfies $y \cle[0]{LR} x$.
   
   Finally, this finishes the proof of the direction $\Rightarrow$ by using
   the characterization of the Kazhdan-Lusztig two-sided cells in terms of the
   shape of the $Q$-symbols.\qedhere
   \end{description}
\end{proof}


\begin{remark}
   Even though Kazhdan-Lusztig left (resp. right) cells coincide with
   left (resp. right) $p$-cells, it is not known whether the corresponding
   preorders coincide as well. Leonardo Patimo has a short proof of the fact
   that the dominance order on partitions is generated by the weak
   Bruhat order. This implies that the Kazhdan-Lusztig two-sided cell preorder
   coincides with the two-sided $p$-cell preorder.
\end{remark}

\cref{thmSymm} implies that the involutions in $S_n$ are precisely those elements $w$
that satisfy $P(w) = Q(w)$. This is the only missing observation for the next
result:

\begin{cor}
   Each left $p$-cell contains a unique involution. Each two-sided $p$-cell 
   contains the longest element in a standard parabolic subgroup.
   Let $C_L$ (resp. $C_R$) be a left (resp. right) $p$-cell. Then we have:
   \[ \lvert C_L \cap C_R \rvert = 
   \begin{cases}
      1 & \text{if } C_L \text{ and } C_R \text{ lie in the same two-sided $p$-cell,}\\
      0 & \text{otherwise.}
   \end{cases} \]
\end{cor}

Let $\pi$ be a permutation of $n$. Recall Frame, Robinson and Thrall's hook length
formula for the number of standard tableaux of shape $\pi$ (\cite[\S4.3]{FuTableaux}):
\[ f^{\pi} = \frac{n!}{\prod_{(i, j) \in \pi} h_{\pi}(i, j)} \]
where $h_{\pi}(i, j)$ denotes the number of boxes in the hook of $(i, j)$ in $\pi$,
i.e. in formulas $h_{\pi}(i, j) = \lvert \{ (a, b) \in \pi \; \vert \; (a = i 
\text{ and } b \geqslant j) \text{ or } (a \geqslant i \text{ and } b = j) \} \rvert$.
The following corollary shows that the hook length formula gives the answer to
some counting problems related to $p$-cells:
\begin{cor}  
   Let $C$ be a two-sided $p$-cell. Denote by $\pi$ the shape of the $P$ and $Q$-symbols 
   of the elements in $C$. Then the following holds:
   \begin{enumerate}
      \item The number of left (or right) $p$-cells in $C$ is given by $f^{\pi}$.
      \item For any left (or right) $p$-cell contained in $C$, the corresponding
            $p$-cell module is free of rank $f^{\pi}$ over $\Z[v, v^{-1}]$.
   \end{enumerate}
\end{cor}

Observe that the preorder $\cle{LR}$ is by definition generated by $\cle{L}$ and $\cle{R}$.
It is in general not clear though whether $\ceq{LR}$ is also generated by $\ceq{L}$
and $\ceq{R}$. For the Kazhdan-Lusztig cell preorder this follows using certain
properties of Lusztig's $a$-function (see \cite[Conjectures 14.2 P9 and P10]{LuUneq}).

\begin{cor}
   The cell preorder $\ceq{LR}$ is generated by $\ceq{L}$ and $\ceq{R}$ in type $A$.
\end{cor}

As a consequence of \cref{lemCellModStar}, we get the following result which is
known for Kazhdan-Lusztig left cell modules (see \cite[Theorem 6.5.2]{BBCoxGrps})
and for which the proof works exactly as in characteristic $0$:

\begin{cor}
   Let $C_1$ and $C_2$ be left $p$-cells in the same two-sided $p$-cell. Then
   the corresponding left cell modules are isomorphic.
\end{cor}

We want to conclude with some interesting questions that merit further
study. We hope to come back to them in another paper:
\begin{itemize}
   \item In type $A$, the $p$-canonical basis for various primes $p$ gives a
         family of interesting bases of each Kazhdan-Lusztig cell module.
         Which bases of the corresponding irreducible representation of $S_n$
         do they specialize to? In \cite[\S1 and \S2.3]{WRedCharVarTypeA} Williamson
         explains that any $p$-canonical basis element of a right (or two-sided)
         cell module that differs from the corresponding Kazhdan-Lusztig basis element
         provides an example of a reducible characteristic variety of a simple
         highest weight module for $sl_{n}(\C)$.
   \item In Type $D$, the generalized $\tau$-invariant can be strengthened to give
         a complete invariant of Kazhdan-Lusztig cells. This is verified in 
         \cite{GaPrimIdealsIV} together with work of McGovern and Pietraho.
         Can one show that all elements of a left $p$-cell have the same strengthened 
         generalized $\tau$-invariant to get that Kazhdan-Lusztig cells decompose 
         into $p$-cells in type $D$?
\end{itemize}

\printbibliography

\Address

\end{document}